\documentclass[11pt,a4paper]{article}
\pdfoutput=1 
\usepackage[margin=1in]{geometry}
\usepackage{amsmath,amsthm,amssymb,enumerate, bbm ,graphicx,color,caption,upgreek, float, tikz, wasysym, subcaption,booktabs,longtable, appendix,graphics, pdfpages,rotating,mathtools,textcomp, array, blkarray, courier, tkz-graph}
 \usetikzlibrary{graphs,graphs.standard}

\usepackage[hidelinks]{hyperref}

\definecolor{blauw}{RGB}{61,158,255}
\definecolor{donkerblauw}{RGB}{0,0,255}
\definecolor{donkergroen}{RGB}{46,148,0}
\definecolor{donkerrood}{RGB}{204,0,0}

\makeatletter 
\newcommand\mynobreakpar{\par\nobreak\@afterheading} 
\makeatother

\makeatletter
\let\@fnsymbol\@arabic
\makeatother

\usepackage[english]{babel}
\newtheorem{conjecture}{Conjecture}[section]
\newtheorem{theorem}{Theorem}[section]
\newtheorem{lemma}[theorem]{Lemma}
\newtheorem{claim}[theorem]{Claim}

\newtheorem{corollary}[theorem]{Corollary}

\theoremstyle{definition}

\newtheorem{examp}{Example}[section]
\newtheorem*{examp*}{Example}

\newcommand{\Forb}{\textrm{Forb}}

\newcommand{\FF}{\mathcal{F}}
\newcommand{\LL}{\mathcal{L}}

\usepackage{tkz-graph}
\tikzset{
  VertexStyle/.append style = {shape=circle,draw, fill=black, minimum size=9pt, inner sep=-1pt},
  EdgeStyle/.append style = {-, thick},
  LoopStyle/.append style = {-}}
\usetikzlibrary{arrows.meta}
\tikzset{>={Latex[width=2.5mm,length=2.5mm]}}

\theoremstyle{plain}
\newfloat{Algorithm}{!hbt}{alg}

\newcounter{thm}[section]

\title{{\large \textbf{SUM-PERFECT GRAPHS}}} \date{}
\author{Bart Litjens\thanks{Korteweg-De Vries Institute for Mathematics, University of Amsterdam. Email: \texttt{bart\_litjens@hotmail.com}, \texttt{s.c.polak@uva.nl}, \texttt{vaidysivaraman@gmail.com}. The research leading to these
results has received funding from the European Research Council under the European Union’s Seventh Framework Programme (FP7/2007-2013) / ERC grant agreement \textnumero 339109.} , Sven Polak\footnotemark[1] , Vaidy Sivaraman\footnotemark[1]}
\begin{document}
\maketitle
\setcounter{footnote}{1}

\noindent \textbf{Abstract.} Inspired by a famous characterization of perfect graphs due to Lov\'{a}sz, we define a graph $G$ to be \textit{sum-perfect} if for every induced subgraph $H$ of $G$, $\alpha(H)  + \omega(H) \geq |V(H)|$. (Here $\alpha$ and $\omega$ denote the stability number and clique number, respectively.) We give a set of $27$ graphs and we prove that a graph $G$ is sum-perfect if and only if $G$ does not contain any of the graphs in the set as an induced subgraph.

\,$\phantom{0}$



\section{Introduction}
All graphs in this article are simple, finite, and undirected. If $G$ is a graph, then $\overline{G}$ denotes its complement. If $G$ and $H$ are graphs, then $G+H$ denotes the disjoint union of $G$ and $H$. We write $nG$ for the disjoint union of $n$ copies of $G$, with $n \geq 1$. For $n \geq 1$, by $P_n$ we denote the path on $n$ vertices and $K_n$ denotes the complete graph on $n$ vertices. For $n \geq 3$, we let $C_n$ denote the cycle on $n$ vertices. \\
\indent Define a graph $G$ to be \emph{sum-perfect} if for every induced subgraph $H$ of $G$, $\alpha(H)  + \omega(H) \geq |V(H)|$. Here, $\alpha$ is the stability number, i.e., the maximum size of a stable set. The parameter $\omega$ denotes the clique number; it is the maximum size of a clique. If~$\LL$ is a set of graphs, we say that~$G$ is~$\LL$-\emph{free} if~$G$ does not contain any graph in~$\LL$ as an induced subgraph. In this paper we prove the following.

\begin{theorem}\label{27th}
A graph~$G$ is sum-perfect if and only if it is~$\FF$-free, where the set ~$\FF:=\{H_1,\ldots,H_{27}\}$ is depicted in Figure~\ref{allgraphs}.
\end{theorem}

\begin{figure}[H]
   \begin{subfigure}{0.137\linewidth}
  \scalebox{0.33}{
\begin{tikzpicture}
\definecolor{cv0}{rgb}{0.0,0.0,0.0}
\definecolor{cfv0}{rgb}{1.0,1.0,1.0}
\definecolor{clv0}{rgb}{0.0,0.0,0.0}
\definecolor{cv1}{rgb}{0.0,0.0,0.0}
\definecolor{cfv1}{rgb}{1.0,1.0,1.0}
\definecolor{clv1}{rgb}{0.0,0.0,0.0}
\definecolor{cv2}{rgb}{0.0,0.0,0.0}
\definecolor{cfv2}{rgb}{1.0,1.0,1.0}
\definecolor{clv2}{rgb}{0.0,0.0,0.0}
\definecolor{cv3}{rgb}{0.0,0.0,0.0}
\definecolor{cfv3}{rgb}{1.0,1.0,1.0}
\definecolor{clv3}{rgb}{0.0,0.0,0.0}
\definecolor{cv4}{rgb}{0.0,0.0,0.0}
\definecolor{cfv4}{rgb}{1.0,1.0,1.0}
\definecolor{clv4}{rgb}{0.0,0.0,0.0}
\definecolor{cv0v1}{rgb}{0.0,0.0,0.0}
\definecolor{cv0v4}{rgb}{0.0,0.0,0.0}
\definecolor{cv1v2}{rgb}{0.0,0.0,0.0}
\definecolor{cv2v3}{rgb}{0.0,0.0,0.0}
\definecolor{cv3v4}{rgb}{0.0,0.0,0.0}
\Vertex[style={minimum size=1.0cm,draw=cv0,fill=cfv0,text=clv0,shape=circle},LabelOut=false,L=\hbox{$0$},x=2.3678cm,y=0.0cm]{v0}
\Vertex[style={minimum size=1.0cm,draw=cv1,fill=cfv1,text=clv1,shape=circle},LabelOut=false,L=\hbox{$1$},x=5.0cm,y=1.5962cm]{v1}
\Vertex[style={minimum size=1.0cm,draw=cv2,fill=cfv2,text=clv2,shape=circle},LabelOut=false,L=\hbox{$2$},x=4.2879cm,y=4.6218cm]{v2}
\Vertex[style={minimum size=1.0cm,draw=cv3,fill=cfv3,text=clv3,shape=circle},LabelOut=false,L=\hbox{$3$},x=1.2053cm,y=5.0cm]{v3}
\Vertex[style={minimum size=1.0cm,draw=cv4,fill=cfv4,text=clv4,shape=circle},LabelOut=false,L=\hbox{$4$},x=0.0cm,y=2.2166cm]{v4}
\Edge[lw=0.1cm,style={color=cv0v1,},](v0)(v1)
\Edge[lw=0.1cm,style={color=cv0v4,},](v0)(v4)
\Edge[lw=0.1cm,style={color=cv1v2,},](v1)(v2)
\Edge[lw=0.1cm,style={color=cv2v3,},](v2)(v3)
\Edge[lw=0.1cm,style={color=cv3v4,},](v3)(v4)
\end{tikzpicture}}
\caption*{$H_1=C_5$}
\end{subfigure}
   \begin{subfigure}{0.137\linewidth}
  \scalebox{0.33}{
\begin{tikzpicture}
\definecolor{cv0}{rgb}{0.0,0.0,0.0}
\definecolor{cfv0}{rgb}{1.0,1.0,1.0}
\definecolor{clv0}{rgb}{0.0,0.0,0.0}
\definecolor{cv1}{rgb}{0.0,0.0,0.0}
\definecolor{cfv1}{rgb}{1.0,1.0,1.0}
\definecolor{clv1}{rgb}{0.0,0.0,0.0}
\definecolor{cv2}{rgb}{0.0,0.0,0.0}
\definecolor{cfv2}{rgb}{1.0,1.0,1.0}
\definecolor{clv2}{rgb}{0.0,0.0,0.0}
\definecolor{cv3}{rgb}{0.0,0.0,0.0}
\definecolor{cfv3}{rgb}{1.0,1.0,1.0}
\definecolor{clv3}{rgb}{0.0,0.0,0.0}
\definecolor{cv4}{rgb}{0.0,0.0,0.0}
\definecolor{cfv4}{rgb}{1.0,1.0,1.0}
\definecolor{clv4}{rgb}{0.0,0.0,0.0}
\definecolor{cv5}{rgb}{0.0,0.0,0.0}
\definecolor{cfv5}{rgb}{1.0,1.0,1.0}
\definecolor{clv5}{rgb}{0.0,0.0,0.0}
\definecolor{cv0v3}{rgb}{0.0,0.0,0.0}
\definecolor{cv1v4}{rgb}{0.0,0.0,0.0}
\definecolor{cv2v5}{rgb}{0.0,0.0,0.0}
\Vertex[style={minimum size=1.0cm,draw=cv0,fill=cfv0,text=clv0,shape=circle},LabelOut=false,L=\hbox{$0$},x=0.0cm,y=5cm]{v0}
\Vertex[style={minimum size=1.0cm,draw=cv1,fill=cfv1,text=clv1,shape=circle},LabelOut=false,L=\hbox{$1$},x=2.5cm,y=0cm]{v1}
\Vertex[style={minimum size=1.0cm,draw=cv2,fill=cfv2,text=clv2,shape=circle},LabelOut=false,L=\hbox{$2$},x=5cm,y=0cm]{v2}
\Vertex[style={minimum size=1.0cm,draw=cv3,fill=cfv3,text=clv3,shape=circle},LabelOut=false,L=\hbox{$3$},x=0.0cm,y=0cm]{v3}
\Vertex[style={minimum size=1.0cm,draw=cv4,fill=cfv4,text=clv4,shape=circle},LabelOut=false,L=\hbox{$4$},x=2.5cm,y=5cm]{v4}
\Vertex[style={minimum size=1.0cm,draw=cv5,fill=cfv5,text=clv5,shape=circle},LabelOut=false,L=\hbox{$5$},x=5.0cm,y=5.0cm]{v5}
\Edge[lw=0.1cm,style={color=cv0v3,},](v0)(v3)
\Edge[lw=0.1cm,style={color=cv1v4,},](v1)(v4)
\Edge[lw=0.1cm,style={color=cv2v5,},](v2)(v5)
\end{tikzpicture}}
\caption*{$H_2=3K_2$}
\end{subfigure}
 \begin{subfigure}{0.137\linewidth}
  \scalebox{0.33}{
\begin{tikzpicture}
\definecolor{cv0}{rgb}{0.0,0.0,0.0}
\definecolor{cfv0}{rgb}{1.0,1.0,1.0}
\definecolor{clv0}{rgb}{0.0,0.0,0.0}
\definecolor{cv1}{rgb}{0.0,0.0,0.0}
\definecolor{cfv1}{rgb}{1.0,1.0,1.0}
\definecolor{clv1}{rgb}{0.0,0.0,0.0}
\definecolor{cv2}{rgb}{0.0,0.0,0.0}
\definecolor{cfv2}{rgb}{1.0,1.0,1.0}
\definecolor{clv2}{rgb}{0.0,0.0,0.0}
\definecolor{cv3}{rgb}{0.0,0.0,0.0}
\definecolor{cfv3}{rgb}{1.0,1.0,1.0}
\definecolor{clv3}{rgb}{0.0,0.0,0.0}
\definecolor{cv4}{rgb}{0.0,0.0,0.0}
\definecolor{cfv4}{rgb}{1.0,1.0,1.0}
\definecolor{clv4}{rgb}{0.0,0.0,0.0}
\definecolor{cv5}{rgb}{0.0,0.0,0.0}
\definecolor{cfv5}{rgb}{1.0,1.0,1.0}
\definecolor{clv5}{rgb}{0.0,0.0,0.0}
\definecolor{cv0v3}{rgb}{0.0,0.0,0.0}
\definecolor{cv0v5}{rgb}{0.0,0.0,0.0}
\definecolor{cv1v4}{rgb}{0.0,0.0,0.0}
\definecolor{cv2v5}{rgb}{0.0,0.0,0.0}
\Vertex[style={minimum size=1.0cm,draw=cv0,fill=cfv0,text=clv0,shape=circle},LabelOut=false,L=\hbox{$0$},x=0.6365cm,y=3.0536cm]{v0}
\Vertex[style={minimum size=1.0cm,draw=cv1,fill=cfv1,text=clv1,shape=circle},LabelOut=false,L=\hbox{$1$},x=3.2355cm,y=0.0cm]{v1}
\Vertex[style={minimum size=1.0cm,draw=cv2,fill=cfv2,text=clv2,shape=circle},LabelOut=false,L=\hbox{$2$},x=1.9878cm,y=0.8935cm]{v2}
\Vertex[style={minimum size=1.0cm,draw=cv3,fill=cfv3,text=clv3,shape=circle},LabelOut=false,L=\hbox{$3$},x=0.0cm,y=4.1815cm]{v3}
\Vertex[style={minimum size=1.0cm,draw=cv4,fill=cfv4,text=clv4,shape=circle},LabelOut=false,L=\hbox{$4$},x=5.0cm,y=5.0cm]{v4}
\Vertex[style={minimum size=1.0cm,draw=cv5,fill=cfv5,text=clv5,shape=circle},LabelOut=false,L=\hbox{$5$},x=1.324cm,y=1.8714cm]{v5}
\Edge[lw=0.1cm,style={color=cv0v3,},](v0)(v3)
\Edge[lw=0.1cm,style={color=cv0v5,},](v0)(v5)
\Edge[lw=0.1cm,style={color=cv1v4,},](v1)(v4)
\Edge[lw=0.1cm,style={color=cv2v5,},](v2)(v5)
\end{tikzpicture}}
\caption*{$H_3=P_4 + K_2$}
\end{subfigure}
 \begin{subfigure}{0.137\linewidth}
  \scalebox{0.33}{
\begin{tikzpicture}
\definecolor{cv0}{rgb}{0.0,0.0,0.0}
\definecolor{cfv0}{rgb}{1.0,1.0,1.0}
\definecolor{clv0}{rgb}{0.0,0.0,0.0}
\definecolor{cv1}{rgb}{0.0,0.0,0.0}
\definecolor{cfv1}{rgb}{1.0,1.0,1.0}
\definecolor{clv1}{rgb}{0.0,0.0,0.0}
\definecolor{cv2}{rgb}{0.0,0.0,0.0}
\definecolor{cfv2}{rgb}{1.0,1.0,1.0}
\definecolor{clv2}{rgb}{0.0,0.0,0.0}
\definecolor{cv3}{rgb}{0.0,0.0,0.0}
\definecolor{cfv3}{rgb}{1.0,1.0,1.0}
\definecolor{clv3}{rgb}{0.0,0.0,0.0}
\definecolor{cv4}{rgb}{0.0,0.0,0.0}
\definecolor{cfv4}{rgb}{1.0,1.0,1.0}
\definecolor{clv4}{rgb}{0.0,0.0,0.0}
\definecolor{cv5}{rgb}{0.0,0.0,0.0}
\definecolor{cfv5}{rgb}{1.0,1.0,1.0}
\definecolor{clv5}{rgb}{0.0,0.0,0.0}
\definecolor{cv0v3}{rgb}{0.0,0.0,0.0}
\definecolor{cv0v5}{rgb}{0.0,0.0,0.0}
\definecolor{cv1v4}{rgb}{0.0,0.0,0.0}
\definecolor{cv1v5}{rgb}{0.0,0.0,0.0}
\definecolor{cv2v5}{rgb}{0.0,0.0,0.0}
\Vertex[style={minimum size=1.0cm,draw=cv0,fill=cfv0,text=clv0,shape=circle},LabelOut=false,L=\hbox{$0$},x=3.6735cm,y=1.7777cm]{v0}
\Vertex[style={minimum size=1.0cm,draw=cv1,fill=cfv1,text=clv1,shape=circle},LabelOut=false,L=\hbox{$1$},x=0.986cm,y=3.6396cm]{v1}
\Vertex[style={minimum size=1.0cm,draw=cv2,fill=cfv2,text=clv2,shape=circle},LabelOut=false,L=\hbox{$2$},x=1.7714cm,y=0.0cm]{v2}
\Vertex[style={minimum size=1.0cm,draw=cv3,fill=cfv3,text=clv3,shape=circle},LabelOut=false,L=\hbox{$3$},x=5.0cm,y=1.5981cm]{v3}
\Vertex[style={minimum size=1.0cm,draw=cv4,fill=cfv4,text=clv4,shape=circle},LabelOut=false,L=\hbox{$4$},x=0.0cm,y=5.0cm]{v4}
\Vertex[style={minimum size=1.0cm,draw=cv5,fill=cfv5,text=clv5,shape=circle},LabelOut=false,L=\hbox{$5$},x=2.1346cm,y=1.9626cm]{v5}
\Edge[lw=0.1cm,style={color=cv0v3,},](v0)(v3)
\Edge[lw=0.1cm,style={color=cv0v5,},](v0)(v5)
\Edge[lw=0.1cm,style={color=cv1v4,},](v1)(v4)
\Edge[lw=0.1cm,style={color=cv1v5,},](v1)(v5)
\Edge[lw=0.1cm,style={color=cv2v5,},](v2)(v5)
\end{tikzpicture}}
\caption*{$H_4=S_{1,2,2}$}
\end{subfigure}
 \begin{subfigure}{0.137\linewidth}
  \scalebox{0.33}{
 \begin{tikzpicture}
\definecolor{cv0}{rgb}{0.0,0.0,0.0}
\definecolor{cfv0}{rgb}{1.0,1.0,1.0}
\definecolor{clv0}{rgb}{0.0,0.0,0.0}
\definecolor{cv1}{rgb}{0.0,0.0,0.0}
\definecolor{cfv1}{rgb}{1.0,1.0,1.0}
\definecolor{clv1}{rgb}{0.0,0.0,0.0}
\definecolor{cv2}{rgb}{0.0,0.0,0.0}
\definecolor{cfv2}{rgb}{1.0,1.0,1.0}
\definecolor{clv2}{rgb}{0.0,0.0,0.0}
\definecolor{cv3}{rgb}{0.0,0.0,0.0}
\definecolor{cfv3}{rgb}{1.0,1.0,1.0}
\definecolor{clv3}{rgb}{0.0,0.0,0.0}
\definecolor{cv4}{rgb}{0.0,0.0,0.0}
\definecolor{cfv4}{rgb}{1.0,1.0,1.0}
\definecolor{clv4}{rgb}{0.0,0.0,0.0}
\definecolor{cv5}{rgb}{0.0,0.0,0.0}
\definecolor{cfv5}{rgb}{1.0,1.0,1.0}
\definecolor{clv5}{rgb}{0.0,0.0,0.0}
\definecolor{cv0v3}{rgb}{0.0,0.0,0.0}
\definecolor{cv1v4}{rgb}{0.0,0.0,0.0}
\definecolor{cv1v5}{rgb}{0.0,0.0,0.0}
\definecolor{cv2v4}{rgb}{0.0,0.0,0.0}
\definecolor{cv2v5}{rgb}{0.0,0.0,0.0}
\Vertex[style={minimum size=1.0cm,draw=cv0,fill=cfv0,text=clv0,shape=circle},LabelOut=false,L=\hbox{$0$},x=5.0cm,y=0.0cm]{v0}
\Vertex[style={minimum size=1.0cm,draw=cv1,fill=cfv1,text=clv1,shape=circle},LabelOut=false,L=\hbox{$1$},x=3.655cm,y=1.5823cm]{v1}
\Vertex[style={minimum size=1.0cm,draw=cv2,fill=cfv2,text=clv2,shape=circle},LabelOut=false,L=\hbox{$2$},x=0.0cm,y=3.4691cm]{v2}
\Vertex[style={minimum size=1.0cm,draw=cv3,fill=cfv3,text=clv3,shape=circle},LabelOut=false,L=\hbox{$3$},x=4.2121cm,y=5.0cm]{v3}
\Vertex[style={minimum size=1.0cm,draw=cv4,fill=cfv4,text=clv4,shape=circle},LabelOut=false,L=\hbox{$4$},x=2.5818cm,y=4.8218cm]{v4}
\Vertex[style={minimum size=1.0cm,draw=cv5,fill=cfv5,text=clv5,shape=circle},LabelOut=false,L=\hbox{$5$},x=1.1081cm,y=0.1267cm]{v5}
\Edge[lw=0.1cm,style={color=cv0v3,},](v0)(v3)
\Edge[lw=0.1cm,style={color=cv1v4,},](v1)(v4)
\Edge[lw=0.1cm,style={color=cv1v5,},](v1)(v5)
\Edge[lw=0.1cm,style={color=cv2v4,},](v2)(v4)
\Edge[lw=0.1cm,style={color=cv2v5,},](v2)(v5)
\end{tikzpicture}}
\caption*{$H_5=C_4+K_2$}
\end{subfigure}
 \begin{subfigure}{0.137\linewidth}
  \scalebox{0.33}{
\begin{tikzpicture}
\definecolor{cv0}{rgb}{0.0,0.0,0.0}
\definecolor{cfv0}{rgb}{1.0,1.0,1.0}
\definecolor{clv0}{rgb}{0.0,0.0,0.0}
\definecolor{cv1}{rgb}{0.0,0.0,0.0}
\definecolor{cfv1}{rgb}{1.0,1.0,1.0}
\definecolor{clv1}{rgb}{0.0,0.0,0.0}
\definecolor{cv2}{rgb}{0.0,0.0,0.0}
\definecolor{cfv2}{rgb}{1.0,1.0,1.0}
\definecolor{clv2}{rgb}{0.0,0.0,0.0}
\definecolor{cv3}{rgb}{0.0,0.0,0.0}
\definecolor{cfv3}{rgb}{1.0,1.0,1.0}
\definecolor{clv3}{rgb}{0.0,0.0,0.0}
\definecolor{cv4}{rgb}{0.0,0.0,0.0}
\definecolor{cfv4}{rgb}{1.0,1.0,1.0}
\definecolor{clv4}{rgb}{0.0,0.0,0.0}
\definecolor{cv5}{rgb}{0.0,0.0,0.0}
\definecolor{cfv5}{rgb}{1.0,1.0,1.0}
\definecolor{clv5}{rgb}{0.0,0.0,0.0}
\definecolor{cv0v3}{rgb}{0.0,0.0,0.0}
\definecolor{cv0v5}{rgb}{0.0,0.0,0.0}
\definecolor{cv1v4}{rgb}{0.0,0.0,0.0}
\definecolor{cv1v5}{rgb}{0.0,0.0,0.0}
\definecolor{cv2v4}{rgb}{0.0,0.0,0.0}
\Vertex[style={minimum size=1.0cm,draw=cv0,fill=cfv0,text=clv0,shape=circle},LabelOut=false,L=\hbox{$0$},x=4.2418cm,y=0.9385cm]{v0}
\Vertex[style={minimum size=1.0cm,draw=cv1,fill=cfv1,text=clv1,shape=circle},LabelOut=false,L=\hbox{$1$},x=2.1685cm,y=3.0915cm]{v1}
\Vertex[style={minimum size=1.0cm,draw=cv2,fill=cfv2,text=clv2,shape=circle},LabelOut=false,L=\hbox{$2$},x=0.0cm,y=5.0cm]{v2}
\Vertex[style={minimum size=1.0cm,draw=cv3,fill=cfv3,text=clv3,shape=circle},LabelOut=false,L=\hbox{$3$},x=5.0cm,y=0.0cm]{v3}
\Vertex[style={minimum size=1.0cm,draw=cv4,fill=cfv4,text=clv4,shape=circle},LabelOut=false,L=\hbox{$4$},x=1.0078cm,y=4.1348cm]{v4}
\Vertex[style={minimum size=1.0cm,draw=cv5,fill=cfv5,text=clv5,shape=circle},LabelOut=false,L=\hbox{$5$},x=3.2709cm,y=2.0112cm]{v5}
\Edge[lw=0.1cm,style={color=cv0v3,},](v0)(v3)
\Edge[lw=0.1cm,style={color=cv0v5,},](v0)(v5)
\Edge[lw=0.1cm,style={color=cv1v4,},](v1)(v4)
\Edge[lw=0.1cm,style={color=cv1v5,},](v1)(v5)
\Edge[lw=0.1cm,style={color=cv2v4,},](v2)(v4)
\end{tikzpicture}}
\caption*{$H_6=P_6$}
\end{subfigure} 
 \begin{subfigure}{0.137\linewidth}
  \scalebox{0.33}{
\begin{tikzpicture}
\definecolor{cv0}{rgb}{0.0,0.0,0.0}
\definecolor{cfv0}{rgb}{1.0,1.0,1.0}
\definecolor{clv0}{rgb}{0.0,0.0,0.0}
\definecolor{cv1}{rgb}{0.0,0.0,0.0}
\definecolor{cfv1}{rgb}{1.0,1.0,1.0}
\definecolor{clv1}{rgb}{0.0,0.0,0.0}
\definecolor{cv2}{rgb}{0.0,0.0,0.0}
\definecolor{cfv2}{rgb}{1.0,1.0,1.0}
\definecolor{clv2}{rgb}{0.0,0.0,0.0}
\definecolor{cv3}{rgb}{0.0,0.0,0.0}
\definecolor{cfv3}{rgb}{1.0,1.0,1.0}
\definecolor{clv3}{rgb}{0.0,0.0,0.0}
\definecolor{cv4}{rgb}{0.0,0.0,0.0}
\definecolor{cfv4}{rgb}{1.0,1.0,1.0}
\definecolor{clv4}{rgb}{0.0,0.0,0.0}
\definecolor{cv5}{rgb}{0.0,0.0,0.0}
\definecolor{cfv5}{rgb}{1.0,1.0,1.0}
\definecolor{clv5}{rgb}{0.0,0.0,0.0}
\definecolor{cv0v3}{rgb}{0.0,0.0,0.0}
\definecolor{cv0v5}{rgb}{0.0,0.0,0.0}
\definecolor{cv1v4}{rgb}{0.0,0.0,0.0}
\definecolor{cv1v5}{rgb}{0.0,0.0,0.0}
\definecolor{cv2v4}{rgb}{0.0,0.0,0.0}
\definecolor{cv2v5}{rgb}{0.0,0.0,0.0}
\Vertex[style={minimum size=1.0cm,draw=cv0,fill=cfv0,text=clv0,shape=circle},LabelOut=false,L=\hbox{$0$},x=3.7462cm,y=1.3599cm]{v0}
\Vertex[style={minimum size=1.0cm,draw=cv1,fill=cfv1,text=clv1,shape=circle},LabelOut=false,L=\hbox{$1$},x=0.6374cm,y=3.0514cm]{v1}
\Vertex[style={minimum size=1.0cm,draw=cv2,fill=cfv2,text=clv2,shape=circle},LabelOut=false,L=\hbox{$2$},x=1.3493cm,y=5.0cm]{v2}
\Vertex[style={minimum size=1.0cm,draw=cv3,fill=cfv3,text=clv3,shape=circle},LabelOut=false,L=\hbox{$3$},x=5.0cm,y=0.0cm]{v3}
\Vertex[style={minimum size=1.0cm,draw=cv4,fill=cfv4,text=clv4,shape=circle},LabelOut=false,L=\hbox{$4$},x=0.0cm,y=4.9229cm]{v4}
\Vertex[style={minimum size=1.0cm,draw=cv5,fill=cfv5,text=clv5,shape=circle},LabelOut=false,L=\hbox{$5$},x=2.1805cm,y=2.937cm]{v5}
\Edge[lw=0.1cm,style={color=cv0v3,},](v0)(v3)
\Edge[lw=0.1cm,style={color=cv0v5,},](v0)(v5)
\Edge[lw=0.1cm,style={color=cv1v4,},](v1)(v4)
\Edge[lw=0.1cm,style={color=cv1v5,},](v1)(v5)
\Edge[lw=0.1cm,style={color=cv2v4,},](v2)(v4)
\Edge[lw=0.1cm,style={color=cv2v5,},](v2)(v5)
\end{tikzpicture}}
\caption*{$H_7$}
\end{subfigure}\par\medskip
 \begin{subfigure}{0.137\linewidth}
  \scalebox{0.33}{
 \begin{tikzpicture}
\definecolor{cv0}{rgb}{0.0,0.0,0.0}
\definecolor{cfv0}{rgb}{1.0,1.0,1.0}
\definecolor{clv0}{rgb}{0.0,0.0,0.0}
\definecolor{cv1}{rgb}{0.0,0.0,0.0}
\definecolor{cfv1}{rgb}{1.0,1.0,1.0}
\definecolor{clv1}{rgb}{0.0,0.0,0.0}
\definecolor{cv2}{rgb}{0.0,0.0,0.0}
\definecolor{cfv2}{rgb}{1.0,1.0,1.0}
\definecolor{clv2}{rgb}{0.0,0.0,0.0}
\definecolor{cv3}{rgb}{0.0,0.0,0.0}
\definecolor{cfv3}{rgb}{1.0,1.0,1.0}
\definecolor{clv3}{rgb}{0.0,0.0,0.0}
\definecolor{cv4}{rgb}{0.0,0.0,0.0}
\definecolor{cfv4}{rgb}{1.0,1.0,1.0}
\definecolor{clv4}{rgb}{0.0,0.0,0.0}
\definecolor{cv5}{rgb}{0.0,0.0,0.0}
\definecolor{cfv5}{rgb}{1.0,1.0,1.0}
\definecolor{clv5}{rgb}{0.0,0.0,0.0}
\definecolor{cv0v3}{rgb}{0.0,0.0,0.0}
\definecolor{cv0v4}{rgb}{0.0,0.0,0.0}
\definecolor{cv1v3}{rgb}{0.0,0.0,0.0}
\definecolor{cv1v5}{rgb}{0.0,0.0,0.0}
\definecolor{cv2v4}{rgb}{0.0,0.0,0.0}
\definecolor{cv2v5}{rgb}{0.0,0.0,0.0}
\Vertex[style={minimum size=1.0cm,draw=cv0,fill=cfv0,text=clv0,shape=circle},LabelOut=false,L=\hbox{$0$},x=0.2328cm,y=4.0716cm]{v0}
\Vertex[style={minimum size=1.0cm,draw=cv1,fill=cfv1,text=clv1,shape=circle},LabelOut=false,L=\hbox{$1$},x=2.2601cm,y=0.0cm]{v1}
\Vertex[style={minimum size=1.0cm,draw=cv2,fill=cfv2,text=clv2,shape=circle},LabelOut=false,L=\hbox{$2$},x=5.0cm,y=3.4413cm]{v2}
\Vertex[style={minimum size=1.0cm,draw=cv3,fill=cfv3,text=clv3,shape=circle},LabelOut=false,L=\hbox{$3$},x=0.0cm,y=1.5618cm]{v3}
\Vertex[style={minimum size=1.0cm,draw=cv4,fill=cfv4,text=clv4,shape=circle},LabelOut=false,L=\hbox{$4$},x=2.8262cm,y=5.0cm]{v4}
\Vertex[style={minimum size=1.0cm,draw=cv5,fill=cfv5,text=clv5,shape=circle},LabelOut=false,L=\hbox{$5$},x=4.7728cm,y=0.9447cm]{v5}
\Edge[lw=0.1cm,style={color=cv0v3,},](v0)(v3)
\Edge[lw=0.1cm,style={color=cv0v4,},](v0)(v4)
\Edge[lw=0.1cm,style={color=cv1v3,},](v1)(v3)
\Edge[lw=0.1cm,style={color=cv1v5,},](v1)(v5)
\Edge[lw=0.1cm,style={color=cv2v4,},](v2)(v4)
\Edge[lw=0.1cm,style={color=cv2v5,},](v2)(v5)
\end{tikzpicture}}
\caption*{$H_8=C_6$}
\end{subfigure}
  \begin{subfigure}{0.137\linewidth}
  \scalebox{0.33}{
\begin{tikzpicture}
\definecolor{cv0}{rgb}{0.0,0.0,0.0}
\definecolor{cfv0}{rgb}{1.0,1.0,1.0}
\definecolor{clv0}{rgb}{0.0,0.0,0.0}
\definecolor{cv1}{rgb}{0.0,0.0,0.0}
\definecolor{cfv1}{rgb}{1.0,1.0,1.0}
\definecolor{clv1}{rgb}{0.0,0.0,0.0}
\definecolor{cv2}{rgb}{0.0,0.0,0.0}
\definecolor{cfv2}{rgb}{1.0,1.0,1.0}
\definecolor{clv2}{rgb}{0.0,0.0,0.0}
\definecolor{cv3}{rgb}{0.0,0.0,0.0}
\definecolor{cfv3}{rgb}{1.0,1.0,1.0}
\definecolor{clv3}{rgb}{0.0,0.0,0.0}
\definecolor{cv4}{rgb}{0.0,0.0,0.0}
\definecolor{cfv4}{rgb}{1.0,1.0,1.0}
\definecolor{clv4}{rgb}{0.0,0.0,0.0}
\definecolor{cv5}{rgb}{0.0,0.0,0.0}
\definecolor{cfv5}{rgb}{1.0,1.0,1.0}
\definecolor{clv5}{rgb}{0.0,0.0,0.0}
\definecolor{cv0v3}{rgb}{0.0,0.0,0.0}
\definecolor{cv0v4}{rgb}{0.0,0.0,0.0}
\definecolor{cv0v5}{rgb}{0.0,0.0,0.0}
\definecolor{cv1v4}{rgb}{0.0,0.0,0.0}
\definecolor{cv1v5}{rgb}{0.0,0.0,0.0}
\definecolor{cv2v5}{rgb}{0.0,0.0,0.0}
\Vertex[style={minimum size=1.0cm,draw=cv0,fill=cfv0,text=clv0,shape=circle},LabelOut=false,L=\hbox{$0$},x=3.8033cm,y=1.9728cm]{v0}
\Vertex[style={minimum size=1.0cm,draw=cv1,fill=cfv1,text=clv1,shape=circle},LabelOut=false,L=\hbox{$1$},x=2.7134cm,y=5.0cm]{v1}
\Vertex[style={minimum size=1.0cm,draw=cv2,fill=cfv2,text=clv2,shape=circle},LabelOut=false,L=\hbox{$2$},x=0.0cm,y=2.3837cm]{v2}
\Vertex[style={minimum size=1.0cm,draw=cv3,fill=cfv3,text=clv3,shape=circle},LabelOut=false,L=\hbox{$3$},x=5.0cm,y=0.0cm]{v3}
\Vertex[style={minimum size=1.0cm,draw=cv4,fill=cfv4,text=clv4,shape=circle},LabelOut=false,L=\hbox{$4$},x=4.2401cm,y=4.2374cm]{v4}
\Vertex[style={minimum size=1.0cm,draw=cv5,fill=cfv5,text=clv5,shape=circle},LabelOut=false,L=\hbox{$5$},x=1.8903cm,y=2.9163cm]{v5}
\Edge[lw=0.1cm,style={color=cv0v3,},](v0)(v3)
\Edge[lw=0.1cm,style={color=cv0v4,},](v0)(v4)
\Edge[lw=0.1cm,style={color=cv0v5,},](v0)(v5)
\Edge[lw=0.1cm,style={color=cv1v4,},](v1)(v4)
\Edge[lw=0.1cm,style={color=cv1v5,},](v1)(v5)
\Edge[lw=0.1cm,style={color=cv2v5,},](v2)(v5)
\end{tikzpicture}}
\caption*{$H_9$}
\end{subfigure}
   \begin{subfigure}{0.137\linewidth}
  \scalebox{0.33}{
 \begin{tikzpicture}
\definecolor{cv0}{rgb}{0.0,0.0,0.0}
\definecolor{cfv0}{rgb}{1.0,1.0,1.0}
\definecolor{clv0}{rgb}{0.0,0.0,0.0}
\definecolor{cv1}{rgb}{0.0,0.0,0.0}
\definecolor{cfv1}{rgb}{1.0,1.0,1.0}
\definecolor{clv1}{rgb}{0.0,0.0,0.0}
\definecolor{cv2}{rgb}{0.0,0.0,0.0}
\definecolor{cfv2}{rgb}{1.0,1.0,1.0}
\definecolor{clv2}{rgb}{0.0,0.0,0.0}
\definecolor{cv3}{rgb}{0.0,0.0,0.0}
\definecolor{cfv3}{rgb}{1.0,1.0,1.0}
\definecolor{clv3}{rgb}{0.0,0.0,0.0}
\definecolor{cv4}{rgb}{0.0,0.0,0.0}
\definecolor{cfv4}{rgb}{1.0,1.0,1.0}
\definecolor{clv4}{rgb}{0.0,0.0,0.0}
\definecolor{cv5}{rgb}{0.0,0.0,0.0}
\definecolor{cfv5}{rgb}{1.0,1.0,1.0}
\definecolor{clv5}{rgb}{0.0,0.0,0.0}
\definecolor{cv0v3}{rgb}{0.0,0.0,0.0}
\definecolor{cv0v4}{rgb}{0.0,0.0,0.0}
\definecolor{cv0v5}{rgb}{0.0,0.0,0.0}
\definecolor{cv1v3}{rgb}{0.0,0.0,0.0}
\definecolor{cv1v5}{rgb}{0.0,0.0,0.0}
\definecolor{cv2v4}{rgb}{0.0,0.0,0.0}
\definecolor{cv2v5}{rgb}{0.0,0.0,0.0}
\Vertex[style={minimum size=1.0cm,draw=cv0,fill=cfv0,text=clv0,shape=circle},LabelOut=false,L=\hbox{$0$},x=2.9337cm,y=3.5445cm]{v0}
\Vertex[style={minimum size=1.0cm,draw=cv1,fill=cfv1,text=clv1,shape=circle},LabelOut=false,L=\hbox{$1$},x=4.0514cm,y=0.0cm]{v1}
\Vertex[style={minimum size=1.0cm,draw=cv2,fill=cfv2,text=clv2,shape=circle},LabelOut=false,L=\hbox{$2$},x=0.0cm,y=2.9066cm]{v2}
\Vertex[style={minimum size=1.0cm,draw=cv3,fill=cfv3,text=clv3,shape=circle},LabelOut=false,L=\hbox{$3$},x=5.0cm,y=2.0991cm]{v3}
\Vertex[style={minimum size=1.0cm,draw=cv4,fill=cfv4,text=clv4,shape=circle},LabelOut=false,L=\hbox{$4$},x=0.8763cm,y=5.0cm]{v4}
\Vertex[style={minimum size=1.0cm,draw=cv5,fill=cfv5,text=clv5,shape=circle},LabelOut=false,L=\hbox{$5$},x=1.9996cm,y=1.367cm]{v5}
\Edge[lw=0.1cm,style={color=cv0v3,},](v0)(v3)
\Edge[lw=0.1cm,style={color=cv0v4,},](v0)(v4)
\Edge[lw=0.1cm,style={color=cv0v5,},](v0)(v5)
\Edge[lw=0.1cm,style={color=cv1v3,},](v1)(v3)
\Edge[lw=0.1cm,style={color=cv1v5,},](v1)(v5)
\Edge[lw=0.1cm,style={color=cv2v4,},](v2)(v4)
\Edge[lw=0.1cm,style={color=cv2v5,},](v2)(v5)
\end{tikzpicture}}
\caption*{$H_{10}$}
\end{subfigure}
 \begin{subfigure}{0.137\linewidth}
  \scalebox{0.33}{
 \begin{tikzpicture}
\definecolor{cv0}{rgb}{0.0,0.0,0.0}
\definecolor{cfv0}{rgb}{1.0,1.0,1.0}
\definecolor{clv0}{rgb}{0.0,0.0,0.0}
\definecolor{cv1}{rgb}{0.0,0.0,0.0}
\definecolor{cfv1}{rgb}{1.0,1.0,1.0}
\definecolor{clv1}{rgb}{0.0,0.0,0.0}
\definecolor{cv2}{rgb}{0.0,0.0,0.0}
\definecolor{cfv2}{rgb}{1.0,1.0,1.0}
\definecolor{clv2}{rgb}{0.0,0.0,0.0}
\definecolor{cv3}{rgb}{0.0,0.0,0.0}
\definecolor{cfv3}{rgb}{1.0,1.0,1.0}
\definecolor{clv3}{rgb}{0.0,0.0,0.0}
\definecolor{cv4}{rgb}{0.0,0.0,0.0}
\definecolor{cfv4}{rgb}{1.0,1.0,1.0}
\definecolor{clv4}{rgb}{0.0,0.0,0.0}
\definecolor{cv5}{rgb}{0.0,0.0,0.0}
\definecolor{cfv5}{rgb}{1.0,1.0,1.0}
\definecolor{clv5}{rgb}{0.0,0.0,0.0}
\definecolor{cv0v3}{rgb}{0.0,0.0,0.0}
\definecolor{cv0v4}{rgb}{0.0,0.0,0.0}
\definecolor{cv0v5}{rgb}{0.0,0.0,0.0}
\definecolor{cv1v3}{rgb}{0.0,0.0,0.0}
\definecolor{cv1v4}{rgb}{0.0,0.0,0.0}
\definecolor{cv1v5}{rgb}{0.0,0.0,0.0}
\definecolor{cv2v5}{rgb}{0.0,0.0,0.0}
\Vertex[style={minimum size=1.0cm,draw=cv0,fill=cfv0,text=clv0,shape=circle},LabelOut=false,L=\hbox{$0$},x=3.8061cm,y=3.9074cm]{v0}
\Vertex[style={minimum size=1.0cm,draw=cv1,fill=cfv1,text=clv1,shape=circle},LabelOut=false,L=\hbox{$1$},x=3.7634cm,y=1.4154cm]{v1}
\Vertex[style={minimum size=1.0cm,draw=cv2,fill=cfv2,text=clv2,shape=circle},LabelOut=false,L=\hbox{$2$},x=0.0cm,y=3.2337cm]{v2}
\Vertex[style={minimum size=1.0cm,draw=cv3,fill=cfv3,text=clv3,shape=circle},LabelOut=false,L=\hbox{$3$},x=5.0cm,y=5.0cm]{v3}
\Vertex[style={minimum size=1.0cm,draw=cv4,fill=cfv4,text=clv4,shape=circle},LabelOut=false,L=\hbox{$4$},x=4.8888cm,y=0.0cm]{v4}
\Vertex[style={minimum size=1.0cm,draw=cv5,fill=cfv5,text=clv5,shape=circle},LabelOut=false,L=\hbox{$5$},x=1.9221cm,y=2.9594cm]{v5}
\Edge[lw=0.1cm,style={color=cv0v3,},](v0)(v3)
\Edge[lw=0.1cm,style={color=cv0v4,},](v0)(v4)
\Edge[lw=0.1cm,style={color=cv0v5,},](v0)(v5)
\Edge[lw=0.1cm,style={color=cv1v3,},](v1)(v3)
\Edge[lw=0.1cm,style={color=cv1v4,},](v1)(v4)
\Edge[lw=0.1cm,style={color=cv1v5,},](v1)(v5)
\Edge[lw=0.1cm,style={color=cv2v5,},](v2)(v5)
\end{tikzpicture}}
\caption*{$H_{11}$}
\end{subfigure}
 \begin{subfigure}{0.137\linewidth}
  \scalebox{0.33}{
\begin{tikzpicture}
\definecolor{cv0}{rgb}{0.0,0.0,0.0}
\definecolor{cfv0}{rgb}{1.0,1.0,1.0}
\definecolor{clv0}{rgb}{0.0,0.0,0.0}
\definecolor{cv1}{rgb}{0.0,0.0,0.0}
\definecolor{cfv1}{rgb}{1.0,1.0,1.0}
\definecolor{clv1}{rgb}{0.0,0.0,0.0}
\definecolor{cv2}{rgb}{0.0,0.0,0.0}
\definecolor{cfv2}{rgb}{1.0,1.0,1.0}
\definecolor{clv2}{rgb}{0.0,0.0,0.0}
\definecolor{cv3}{rgb}{0.0,0.0,0.0}
\definecolor{cfv3}{rgb}{1.0,1.0,1.0}
\definecolor{clv3}{rgb}{0.0,0.0,0.0}
\definecolor{cv4}{rgb}{0.0,0.0,0.0}
\definecolor{cfv4}{rgb}{1.0,1.0,1.0}
\definecolor{clv4}{rgb}{0.0,0.0,0.0}
\definecolor{cv5}{rgb}{0.0,0.0,0.0}
\definecolor{cfv5}{rgb}{1.0,1.0,1.0}
\definecolor{clv5}{rgb}{0.0,0.0,0.0}
\definecolor{cv0v3}{rgb}{0.0,0.0,0.0}
\definecolor{cv0v4}{rgb}{0.0,0.0,0.0}
\definecolor{cv0v5}{rgb}{0.0,0.0,0.0}
\definecolor{cv1v3}{rgb}{0.0,0.0,0.0}
\definecolor{cv1v4}{rgb}{0.0,0.0,0.0}
\definecolor{cv1v5}{rgb}{0.0,0.0,0.0}
\definecolor{cv2v4}{rgb}{0.0,0.0,0.0}
\definecolor{cv2v5}{rgb}{0.0,0.0,0.0}
\Vertex[style={minimum size=1.0cm,draw=cv0,fill=cfv0,text=clv0,shape=circle},LabelOut=false,L=\hbox{$0$},x=1.475cm,y=3.5267cm]{v0}
\Vertex[style={minimum size=1.0cm,draw=cv1,fill=cfv1,text=clv1,shape=circle},LabelOut=false,L=\hbox{$1$},x=5.0cm,y=2.9799cm]{v1}
\Vertex[style={minimum size=1.0cm,draw=cv2,fill=cfv2,text=clv2,shape=circle},LabelOut=false,L=\hbox{$2$},x=0.7092cm,y=0.0cm]{v2}
\Vertex[style={minimum size=1.0cm,draw=cv3,fill=cfv3,text=clv3,shape=circle},LabelOut=false,L=\hbox{$3$},x=4.6792cm,y=5.0cm]{v3}
\Vertex[style={minimum size=1.0cm,draw=cv4,fill=cfv4,text=clv4,shape=circle},LabelOut=false,L=\hbox{$4$},x=3.9447cm,y=1.4905cm]{v4}
\Vertex[style={minimum size=1.0cm,draw=cv5,fill=cfv5,text=clv5,shape=circle},LabelOut=false,L=\hbox{$5$},x=0.0cm,y=1.9005cm]{v5}
\Edge[lw=0.1cm,style={color=cv0v3,},](v0)(v3)
\Edge[lw=0.1cm,style={color=cv0v4,},](v0)(v4)
\Edge[lw=0.1cm,style={color=cv0v5,},](v0)(v5)
\Edge[lw=0.1cm,style={color=cv1v3,},](v1)(v3)
\Edge[lw=0.1cm,style={color=cv1v4,},](v1)(v4)
\Edge[lw=0.1cm,style={color=cv1v5,},](v1)(v5)
\Edge[lw=0.1cm,style={color=cv2v4,},](v2)(v4)
\Edge[lw=0.1cm,style={color=cv2v5,},](v2)(v5)
\end{tikzpicture}}
\caption*{$H_{12}=K_{3,3}-e$}
\end{subfigure}
 \begin{subfigure}{0.137\linewidth}
  \scalebox{0.33}{
\begin{tikzpicture}
\definecolor{cv0}{rgb}{0.0,0.0,0.0}
\definecolor{cfv0}{rgb}{1.0,1.0,1.0}
\definecolor{clv0}{rgb}{0.0,0.0,0.0}
\definecolor{cv1}{rgb}{0.0,0.0,0.0}
\definecolor{cfv1}{rgb}{1.0,1.0,1.0}
\definecolor{clv1}{rgb}{0.0,0.0,0.0}
\definecolor{cv2}{rgb}{0.0,0.0,0.0}
\definecolor{cfv2}{rgb}{1.0,1.0,1.0}
\definecolor{clv2}{rgb}{0.0,0.0,0.0}
\definecolor{cv3}{rgb}{0.0,0.0,0.0}
\definecolor{cfv3}{rgb}{1.0,1.0,1.0}
\definecolor{clv3}{rgb}{0.0,0.0,0.0}
\definecolor{cv4}{rgb}{0.0,0.0,0.0}
\definecolor{cfv4}{rgb}{1.0,1.0,1.0}
\definecolor{clv4}{rgb}{0.0,0.0,0.0}
\definecolor{cv5}{rgb}{0.0,0.0,0.0}
\definecolor{cfv5}{rgb}{1.0,1.0,1.0}
\definecolor{clv5}{rgb}{0.0,0.0,0.0}
\definecolor{cv0v3}{rgb}{0.0,0.0,0.0}
\definecolor{cv0v4}{rgb}{0.0,0.0,0.0}
\definecolor{cv0v5}{rgb}{0.0,0.0,0.0}
\definecolor{cv1v3}{rgb}{0.0,0.0,0.0}
\definecolor{cv1v4}{rgb}{0.0,0.0,0.0}
\definecolor{cv1v5}{rgb}{0.0,0.0,0.0}
\definecolor{cv2v3}{rgb}{0.0,0.0,0.0}
\definecolor{cv2v4}{rgb}{0.0,0.0,0.0}
\definecolor{cv2v5}{rgb}{0.0,0.0,0.0}
\Vertex[style={minimum
size=1.0cm,draw=cv0,fill=cfv0,text=clv0,shape=circle},LabelOut=false,L=\hbox{$0$},x=0.0cm,y=5.0cm]{v0}
\Vertex[style={minimum
size=1.0cm,draw=cv1,fill=cfv1,text=clv1,shape=circle},LabelOut=false,L=\hbox{$1$},x=2.5cm,y=5.0cm]{v1}
\Vertex[style={minimum
size=1.0cm,draw=cv2,fill=cfv2,text=clv2,shape=circle},LabelOut=false,L=\hbox{$2$},x=5.0cm,y=5.0cm]{v2}
\Vertex[style={minimum
size=1.0cm,draw=cv3,fill=cfv3,text=clv3,shape=circle},LabelOut=false,L=\hbox{$3$},x=0.0cm,y=0.0cm]{v3}
\Vertex[style={minimum
size=1.0cm,draw=cv4,fill=cfv4,text=clv4,shape=circle},LabelOut=false,L=\hbox{$4$},x=2.5cm,y=0.0cm]{v4}
\Vertex[style={minimum
size=1.0cm,draw=cv5,fill=cfv5,text=clv5,shape=circle},LabelOut=false,L=\hbox{$5$},x=5.0cm,y=0.0cm]{v5}
\Edge[lw=0.1cm,style={color=cv0v3,},](v0)(v3)
\Edge[lw=0.1cm,style={color=cv0v4,},](v0)(v4)
\Edge[lw=0.1cm,style={color=cv0v5,},](v0)(v5)
\Edge[lw=0.1cm,style={color=cv1v3,},](v1)(v3)
\Edge[lw=0.1cm,style={color=cv1v4,},](v1)(v4)
\Edge[lw=0.1cm,style={color=cv1v5,},](v1)(v5)
\Edge[lw=0.1cm,style={color=cv2v3,},](v2)(v3)
\Edge[lw=0.1cm,style={color=cv2v4,},](v2)(v4)
\Edge[lw=0.1cm,style={color=cv2v5,},](v2)(v5)
\end{tikzpicture}}
\caption*{$H_{13} = K_{3,3}$}
\end{subfigure}
 \begin{subfigure}{0.137\linewidth}
  \scalebox{0.33}{
\begin{tikzpicture}
\definecolor{cv0}{rgb}{0.0,0.0,0.0}
\definecolor{cfv0}{rgb}{1.0,1.0,1.0}
\definecolor{clv0}{rgb}{0.0,0.0,0.0}
\definecolor{cv1}{rgb}{0.0,0.0,0.0}
\definecolor{cfv1}{rgb}{1.0,1.0,1.0}
\definecolor{clv1}{rgb}{0.0,0.0,0.0}
\definecolor{cv2}{rgb}{0.0,0.0,0.0}
\definecolor{cfv2}{rgb}{1.0,1.0,1.0}
\definecolor{clv2}{rgb}{0.0,0.0,0.0}
\definecolor{cv3}{rgb}{0.0,0.0,0.0}
\definecolor{cfv3}{rgb}{1.0,1.0,1.0}
\definecolor{clv3}{rgb}{0.0,0.0,0.0}
\definecolor{cv4}{rgb}{0.0,0.0,0.0}
\definecolor{cfv4}{rgb}{1.0,1.0,1.0}
\definecolor{clv4}{rgb}{0.0,0.0,0.0}
\definecolor{cv5}{rgb}{0.0,0.0,0.0}
\definecolor{cfv5}{rgb}{1.0,1.0,1.0}
\definecolor{clv5}{rgb}{0.0,0.0,0.0}
\definecolor{cv6}{rgb}{0.0,0.0,0.0}
\definecolor{cfv6}{rgb}{1.0,1.0,1.0}
\definecolor{clv6}{rgb}{0.0,0.0,0.0}
\definecolor{cv0v3}{rgb}{0.0,0.0,0.0}
\definecolor{cv0v5}{rgb}{0.0,0.0,0.0}
\definecolor{cv0v6}{rgb}{0.0,0.0,0.0}
\definecolor{cv1v4}{rgb}{0.0,0.0,0.0}
\definecolor{cv1v5}{rgb}{0.0,0.0,0.0}
\definecolor{cv2v4}{rgb}{0.0,0.0,0.0}
\definecolor{cv2v6}{rgb}{0.0,0.0,0.0}
\definecolor{cv4v5}{rgb}{0.0,0.0,0.0}
\definecolor{cv4v6}{rgb}{0.0,0.0,0.0}
\definecolor{cv5v6}{rgb}{0.0,0.0,0.0}
\Vertex[style={minimum size=1.0cm,draw=cv0,fill=cfv0,text=clv0,shape=circle},LabelOut=false,L=\hbox{$0$},x=1.6399cm,y=1.7777cm]{v0}
\Vertex[style={minimum size=1.0cm,draw=cv1,fill=cfv1,text=clv1,shape=circle},LabelOut=false,L=\hbox{$1$},x=5.0cm,y=4.6271cm]{v1}
\Vertex[style={minimum size=1.0cm,draw=cv2,fill=cfv2,text=clv2,shape=circle},LabelOut=false,L=\hbox{$2$},x=0.0cm,y=5.0cm]{v2}
\Vertex[style={minimum size=1.0cm,draw=cv3,fill=cfv3,text=clv3,shape=circle},LabelOut=false,L=\hbox{$3$},x=1.0615cm,y=0.0cm]{v3}
\Vertex[style={minimum size=1.0cm,draw=cv4,fill=cfv4,text=clv4,shape=circle},LabelOut=false,L=\hbox{$4$},x=2.4726cm,y=4.7027cm]{v4}
\Vertex[style={minimum size=1.0cm,draw=cv5,fill=cfv5,text=clv5,shape=circle},LabelOut=false,L=\hbox{$5$},x=3.2604cm,y=3.3562cm]{v5}
\Vertex[style={minimum size=1.0cm,draw=cv6,fill=cfv6,text=clv6,shape=circle},LabelOut=false,L=\hbox{$6$},x=1.0893cm,y=3.4919cm]{v6}
\Edge[lw=0.1cm,style={color=cv0v3,},](v0)(v3)
\Edge[lw=0.1cm,style={color=cv0v5,},](v0)(v5)
\Edge[lw=0.1cm,style={color=cv0v6,},](v0)(v6)
\Edge[lw=0.1cm,style={color=cv1v4,},](v1)(v4)
\Edge[lw=0.1cm,style={color=cv1v5,},](v1)(v5)
\Edge[lw=0.1cm,style={color=cv2v4,},](v2)(v4)
\Edge[lw=0.1cm,style={color=cv2v6,},](v2)(v6)
\Edge[lw=0.1cm,style={color=cv4v5,},](v4)(v5)
\Edge[lw=0.1cm,style={color=cv4v6,},](v4)(v6)
\Edge[lw=0.1cm,style={color=cv5v6,},](v5)(v6)
\end{tikzpicture}}
\caption*{$H_{26}$}
\end{subfigure}\par\medskip
  \caption{The graphs in~$\FF$. For each graph, also the complement is contained in~$\FF$. For~$i \in \{14,\ldots,25\}$, we define~$H_{i}:=\overline{H_{i-12}}$, and we define~$H_{27}:=\overline{H_{26}}$.}\label{allgraphs}
\end{figure}
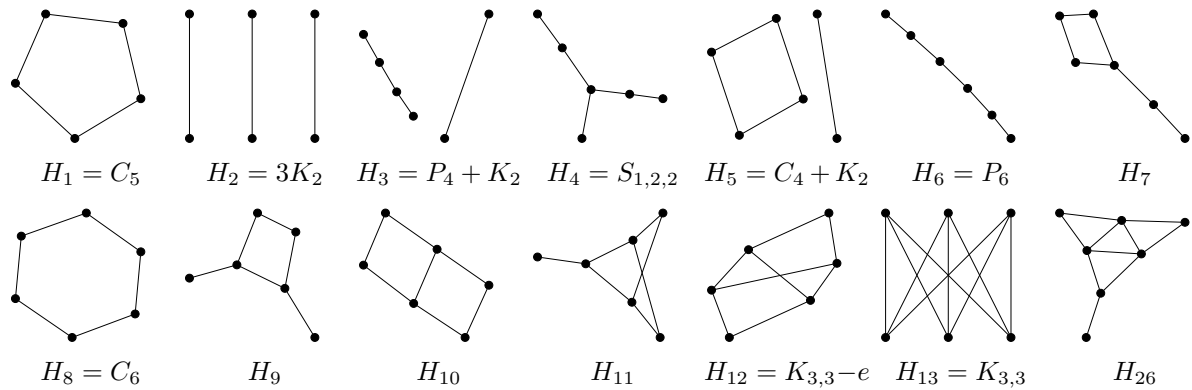

\noindent The set~$\FF$ of Theorem \ref{27th} consists of the 5-cycle $H_1$, all bipartite graphs $H_2,\ldots,H_{13}$ on 6 vertices containing a perfect matching, their complements~$H_{14},\ldots,H_{25}$, and two complementary graphs~$H_{26}$ and~$H_{27}$ on 7 vertices. 

The motivation for studying sum-perfect graphs comes from the following characterization of perfect graphs. Recall that a graph $G$ is perfect if $\chi(H) = \omega(H)$ for all induced subgraphs $H$ of $G$. Here, $\chi$ denotes the chromatic number. Lov\'{a}sz \cite{LL} proved that a graph $G$ is perfect if and only if $\alpha(H) \omega(H) \geq |V(H)|$ for all induced subgraphs $H$ of $G$. When multiplication is replaced by addition, we move from perfect graphs to sum-perfect graphs. Since the condition is strengthened, the class of sum-perfect graphs is a subclass of the class of perfect graphs. 

The strong perfect graph theorem \cite{CRST} asserts that a graph is perfect if and only if it neither has $C_n$ nor $\overline{C_n}$ as an induced subgraph, for odd $n\geq 5$. We determine all the forbidden induced subgraphs for the class of sum-perfect graphs, i.e, graphs that are not sum-perfect but for which every proper induced subgraph is sum-perfect. As~$\alpha(G)=\omega(\overline{G})$ and~$\omega(G)=\alpha(\overline{G})$, the class of sum-perfect graphs is closed under taking the complement.

We give another reason why sum-perfect graphs are interesting. It is based on the following easy observation. 
\begin{align}\label{align:msobservatie}
\text{Let $G$ be a graph, $S$ a stable set and $M$ a clique. Then $|S \cap M| \leq 1$.}
\end{align}
\noindent This implies that $\alpha(G) + \omega(G) \leq |V(G)|+1$.  In Theorem \ref{theorem:threshold} we prove that the class of graphs all of whose induced subgraphs attain this upper bound, coincides with the class of threshold graphs (see Example \ref{example:exasp} for the definition). This yields a characterization of threshold graphs that we could not find in the literature. Fixing $c \geq 0$, the class of graphs $G$ for which $\alpha(H) + \omega(H) = |V(H)|-c$ for every induced subgraph $H$ of $G$, is trivially empty. Indeed, an isolated vertex already does not satisfy this condition. Hence the natural generalization of threshold graphs from this point of view, is to consider graphs $G$ for which $\alpha(H) + \omega(H) \geq |V(H)|-c$ for every induced subgraph $H$ of $G$. For $c = 0$, we obtain the sum-perfect graphs.

\begin{examp}\label{example:exasp}
The following classes of graphs are examples of sum-perfect graphs: split graphs and (apex) threshold graphs. A graph is \emph{split} it its vertex set can be partitioned into a clique and a stable set. A graph is \emph{threshold} if there exists an ordering~$(v_1,\ldots,v_n)$ of the vertices such that each vertex~$v_i$ is either adjacent or non-adjacent to all vertices~$v_1,\ldots,v_{i-1}$. A graph is \emph{apex-threshold} if it contains a vertex whose deletion results in a threshold graph.
\end{examp}

\begin{theorem}\label{theorem:threshold}
A graph $G$ is threshold if and only if for every induced subgraph $H$ of $G$, $\alpha(H) + \omega(H) = |V(H)| + 1$. 
\end{theorem}
\begin{proof}
Let $G$ be a threshold graph. We argue by induction on $|V(G)|$. The base case is trivial. From the definition $G$ either contains a dominating vertex or an isolated vertex. Suppose $G$ has a dominating vertex $v$. Let $H := G - v$. By the induction hypothesis, $\alpha(H) + \omega(H) = |V(H)| + 1$. But $\alpha(G) = \alpha(H)$ and $\omega(G) = \omega(H) + 1$. Hence, $\alpha(G) + \omega(G) = \alpha(H) + \omega(H) + 1 = |V(H)| + 1 + 1 = |V(G)| + 1$. The case when there is an isolated vertex is similar. \\
\indent For the converse, let $G$ be a graph such that every induced subgraph $H$ of $G$ satisfies $\alpha(H) + \omega(H) = |V(H)| + 1$. It is easy to check that the three graphs $P_4, C_4, 2K_2$ fail to satisfy the condition. Hence $G$ contains none of $P_4, C_4, 2K_2$ as an induced subgraph. But a well known theorem tells that a graph is threshold if and only if it contains none of $P_4, C_4, 2K_2$ as an induced subgraph (see \cite{CH73}). Hence $G$ is threshold. 
\end{proof}

\begin{examp}\label{examp:nonexasp}
The class of sum-perfect graphs is contained in the class of weakly chordal graphs. A graph is \emph{weakly chordal} if it contains neither a cycle of length at least $5$ nor its complement as an induced subgraph. 
\end{examp}

With regard to complexity, we note that computing $\alpha(G)+\omega(G)$ of a graph $G$ is $\text{NP}$-hard, as computing the stability number for triangle-free graphs is already $\text{NP}$-hard \cite{P}. 
In Section \ref{section:add}, we briefly address some further optimization problems. Section \ref{section:proof} is devoted to the proof of Theorem \ref{27th}. 

\section{Forbidden induced subgraphs}\label{section:proof}

Let $\Forb$ denote the collection of forbidden induced subgraphs for the class of sum-perfect graphs and let~$\FF$ denote the set of~$27$ graphs from Figure \ref{allgraphs}. If $\LL$ is a collection of graphs, then for $n \geq 1$, let $\LL_n$ denote the graphs in $\LL$ that have $n$ vertices. We define $\LL_{\leq n} := \cup_{i=1}^{n}\LL_i$. In Section \ref{section:inclusion1} we show that~$\Forb_{\leq 7} = \FF$. In Section \ref{section:inclusion2} we prove that $\Forb_n$ is empty for $n \geq 8$. Combining both results yields Theorem \ref{27th}.

\subsection{Forbidden graphs with at most $7$ vertices}\label{section:inclusion1}

With notation as above, we first determine $\Forb_{\leq 5}$ and then show that the graphs in $\FF$ with $6$ vertices are forbidden.

\begin{lemma}\label{lemma:5vertex}
We have that $\mathrm{Forb}_{\leq 5} = \{C_5\}$.
\end{lemma}
\begin{proof}
Let $G \in \Forb_{\leq 5}$. Since complete graphs and empty graphs are sum-perfect, we may assume that $\alpha(G), \omega(G) \geq 2$. Hence $G$ must have $5$ vertices. Suppose that $\alpha(G) = \omega(G) = 2$ (otherwise $\alpha(G) + \omega(G) \geq 5 = |V(G)|$, a contradiction to $G \in \Forb$). No vertex in $G$ can have degree $3$ or $4$, since that would create a triangle either in $G$ or in $\overline{G}$. By the same argument, no vertex can have degree $3$ or $4$ in $\overline{G}$. Hence every vertex in $G$ has degree $2$. But $C_5$ is the only $2$-regular graph on $5$ vertices. 
\end{proof}

\begin{lemma}\label{lemma:6vertex}
We have that $\FF_6 \subseteq \mathrm{Forb}$.
\end{lemma}
\begin{proof}
As $\Forb$ is closed under taking the complement, it suffices to prove that every $6$-vertex bipartite graph with a perfect matching is in $\mathrm{Forb}$. Let $G$ be such a graph. By the previous lemma, every proper induced subgraph of $G$ is sum-perfect. So it remains to prove that $\alpha(G) + \omega(G) < 6$. This is easy, since $\alpha(G) \leq 3$ (because $G$ has a perfect matching) and $\omega(G) = 2$ (because $G$ has no triangles). 
\end{proof}

\noindent Note that there are exactly~$12$ bipartite graphs on~$6$ vertices containing a perfect matching. To see this, start with $3K_2$. There are at most~$6$ additional edges, giving rise to the graphs $H_{2},...,H_{13}$ in Figure~$\ref{allgraphs}$.
%
\noindent One verifies that the 7-vertex graph~$H_{26}$ from Figure~$\ref{allgraphs}$ is in~$\Forb$. Its complement is then automatically forbidden as well. Hence, we have shown that $\FF \subseteq \Forb$.\\
\indent We will derive some properties of minimal non-sum-perfect graphs. The following lemma is key in many of the proofs that will follow. 

\begin{lemma}\label{lemma:removal}
Let $G \in \mathrm{Forb}$. Then for every $v \in V(G)$, we have $\alpha(G - v) = \alpha(G)$ and~$\omega(G - v) = \omega(G)$.
\end{lemma}
\begin{proof}
Let $v \in V(G)$. We know $H := G - v$ is sum-perfect, while~$G$ is not. Hence
$$
 \alpha(H) + \omega(H) \geq |V(H)| = |V(G)| -1 \geq \alpha(G) + \omega(G).
 $$ 
 As~$\omega(H) \leq \omega(G)$ and~$\alpha(H) \leq \alpha(G)$, this yields~$\omega(H) =\omega(G)$ and~$\alpha(H) = \alpha(G)$. 
\end{proof}

\begin{corollary}\label{corollary:n-1}
Let $G \in \mathrm{Forb}$. Then $\alpha(G)+\omega(G) = |V(G)|-1$.
\end{corollary}
\begin{proof}
As $G \in \Forb$, we know that $\alpha(G)+\omega(G) < |V(G)|$. Let $v \in V(G)$ and set $H := G-v$. Then $H$ is sum-perfect and we compute
$$
|V(G)| > \alpha(G) + \omega(G) = \alpha(H) + \omega(H) \geq |V(H)| = |V(G)|-1,
$$
where in the first equality we use Lemma \ref{lemma:removal}. The corollary now follows.
\end{proof}

\begin{lemma}\label{lemma:perfect}
Let $G \in \mathrm{Forb}_{n}$, with $n \geq 6$. Then $G$ is perfect.
\end{lemma}
\begin{proof}
As $G$ is in $\Forb$ and has at least $6$ vertices, $C_5$ is not an induced subgraph of $G$. Assume $G$ is neither $P_6$ nor its complement (otherwise we are done). Then it does not contain $C_n$ or $\overline{C_n}$ as an induced subgraph, for odd $n \geq 7$. Indeed, $C_n$ contains $P_6$ as an induced subgraph for $n \geq 7$, and $\overline{C_n}$ contains $\overline{P_6}$ as an induced subgraph for $n \geq 7$. By the strong perfect graph theorem, we are done.
\end{proof}

Let $G = (V,E)$ be a graph. A \textit{vertex cover} of $G$ is a subset $U \subseteq V$ such that $U \cap e \neq \emptyset$, for each $e \in E$. It is clear that $U \subseteq V$ is a vertex cover if and only if $V\setminus U$ is a stable set. Then 
\begin{equation}\label{equation:gallai}
\alpha(G) + \tau(G) = |V(G)|,
\end{equation}
where $\tau(G)$ denotes the \textit{vertex cover number}, i.e., the minimum size of a vertex cover. Let $\nu(G)$ denote the \textit{matching number}: it is the maximum size of a matching in $G$.

\begin{lemma}\label{lemma:6}
We have that $\mathrm{Forb}_{\leq 6} = \FF_{\leq 6}$.
\end{lemma}
\begin{proof}
The ``$\supseteq$" inclusion follows from Lemmas \ref{lemma:5vertex} and \ref{lemma:6vertex}. To show ``$\subseteq$", let $G \in \Forb_{6}$. As $G$ is in $\Forb$, we have $\alpha(G), \omega(G) \geq 2$. By Corollary \ref{corollary:n-1}, without loss of generality we may assume that $\alpha(G) = 3$ and $\omega(G) = 2$. Equation (\ref{equation:gallai}) gives $\tau(G) = |V(G)| - \alpha(G) = 3$. Lemma \ref{lemma:perfect} implies that $G$ is perfect. Hence $\chi(G) = \omega(G) = 2$, showing that $G$ is bipartite. By K\"onig's theorem $\nu(G) = \tau(G) = 3$ (Theorem $2.1.1$ in \cite{RD}), therefore $G \in \FF_6$.
\end{proof}

\begin{lemma}\label{lemma:disconnected}
The only disconnected members of $\mathrm{Forb}$ are $3K_2, P_4 + K_2, C_4 + K_2, 2K_3 (=\overline{K_{3,3}})$. In particular, $G \in \mathrm{Forb}_{n}$, with $n \geq 7$, implies that $G$ is connected.
\end{lemma}
\begin{proof}
By Lemma \ref{lemma:6vertex} we have that $3K_2, P_4 + K_2, C_4 + K_2, 2K_3 \in \Forb$. Let $G$ be a disconnected graph in $\Forb$. Note that $G$ has no isolated vertices. Suppose $G$ has at least $3$ components. Each component must have an edge, and hence~$G$ contains $3K_2$. Thus the only graph in $\Forb$ with more than $2$ components is $3K_2$. Now, let $G$ be a graph in $\Forb$ with exactly $2$ components. Suppose one of the components is not a threshold graph. Then it contains either $P_4, C_4$, or $2K_2$ (see \cite{CH73}). Together with an edge in the other component, we get one of $3K_2, P_4 + K_2, C_4 + K_2$. Hence both the components are threshold graphs. One of them must be a star, for otherwise we obtain $2K_3$. Hence $G$ is the disjoint union of a threshold graph with a star, which is apex-threshold, and hence sum-perfect, a contradiction. 
\end{proof}

\noindent Before we move on to the case $n=7$, we need a lemma.

\begin{lemma}\label{lemma:outercase}
Let $G \in \mathrm{Forb}_{n}$, with $n \geq 7$. Then $\omega(G) \neq 2$ (and hence also $\alpha(G) \neq 2$).
\end{lemma}
\begin{proof}
Assume that $\omega(G) = 2$. Then $\alpha(G) = n-3$ by Corollary \ref{corollary:n-1}. By Lemma \ref{lemma:perfect}, $G$ is perfect. Hence $\chi(G) = \omega(G) = 2$, showing that $G$ is bipartite. Lemma \ref{lemma:disconnected} implies that $G$ is connected. As $\alpha(G) = n-3$, one of its color classes then has size $n-3$ and the other has size $3$. By K\"onig's theorem, $\nu(G) =3$. As $n \geq 7$, removing any vertex not in the edges of a maximum size matching results in a graph $G'$ with $\alpha(G') = n-4$, contradicting Lemma \ref{lemma:removal}.
\end{proof}

\begin{lemma}\label{lemma:7}
We have that $\mathrm{Forb}_7 = \FF_7$.
\end{lemma}
\begin{proof}
The ``$\supseteq$" inclusion is immediately verified. To show ``$\subseteq$", let $G = (V,E) \in \Forb_7$. By Lemma \ref{lemma:disconnected}, $G$ is connected. By going to the complement if necessary, 
\begin{equation}\label{equation:prop0}
|E| \leq 10.
\end{equation}
Then we must prove that $G = H_{26}$, from Figure \ref{allgraphs}. Corollary \ref{corollary:n-1} and Lemma \ref{lemma:outercase} show that 
\begin{equation}\label{equation:prop1}
\alpha(G) = \omega(G) = 3. 
\end{equation}
Hence there is a triangle in $G$. If there is a vertex $v \in V$ that is contained in all triangles of $G$, then $\omega(G-v) < \omega(G)$, contradicting Lemma \ref{lemma:removal}. Hence 
\begin{align}\label{equation:prop2}
\text{no vertex is contained in all triangles (hence by (\ref{equation:prop1}) there is no dominating vertex).}
\end{align}
If there are two vertex-disjoint triangles and if $v$ denotes the vertex not in the vertex-disjoint triangles, then $\alpha(G-v) < \alpha(G)$, contradicting again Lemma \ref{lemma:removal}. Hence 
\begin{align}\label{equation:prop3}
\text{there are no two vertex-disjoint triangles.}
\end{align}
\indent We introduce some notation. For $v \in V$, let $N(v)$ denote the set of neighbors of $v$. By $N[v]$ we denote the set $N(v) \cup \{v\}$. For $v \in V$, let $d_{\Delta}(v)$ denote the number of triangles in which $v$ is contained and let $d(v)$ denote the degree of $v$.\\\\
\noindent \textit{Claim $1$: For $v \in V$, $d_{\Delta}(v) \leq 3$.}\\
\noindent Suppose $d_{\Delta}(v) = 4.$ Clearly then $d(v) > 3$. If $d(v) = 4$, then $N(v)$ induces a $4$-cycle and in $G\setminus N[v]$ there are two more vertices and at most two more edges. We see that $v$ is in every triangle, contradicting (\ref{equation:prop2}). Suppose $d(v) = 5$. The graph induced by $N[v]$ has $9$ edges, and again we see that $v$ is in every triangle, a contradiction. \\
\indent Suppose $d_{\Delta}(v)  \geq 5.$ Then $d(v) >  4$, so $d(v) = 5$ (as there is no dominating vertex by (\ref{equation:prop2})). The graph induced by $N[v]$ has at least $10$ edges, and so the vertex not in $N[v]$ is isolated, contradicting connectedness of $G$. This proves the claim.\\\\
\noindent \textit{Claim $2$: For $v \in V$, $d_{\Delta}(v) < d(v)$.}\\
\noindent Let $d_{\Delta}(v) = i$. By claim $1$ we know $i \leq 3$. If $i=0$ the claim follows from the fact that $G$ is connected. The cases $i=1$ and $i=2$ are clear. The case $i=3$ follows from $G$ being $K_4$-free (because of (\ref{equation:prop1})), thus proving the claim.\\\\
\noindent \textit{Claim $3$: $G$ has at most $4$ triangles.}\\
We have $20 \geq 2|E| = \sum_{v \in V}d(v) \geq  \sum_{v \in V}(d_{\Delta}(v) + 1).$ Here we are using (\ref{equation:prop0}), the degree-sum formula and claim $2$. We conclude that  $3t(G) = \sum_{v \in V}d_{\Delta}(v) \leq 13$, where $t(G)$ is the number of triangles in $G$. Hence $G$ has at most $4$ triangles. \\\\
\noindent \textit{Claim $4$: $G$ has at least $4$ triangles.}\\
By (\ref{equation:prop2}) and (\ref{equation:prop3}) $G$ cannot have precisely $1$ triangle or precisely $2$ triangles. Suppose there are precisely $3$ triangles. Start with two triangles having a vertex $v$ in common. A third triangle cannot contain the vertex $v$ because of (\ref{equation:prop2}), and hence must contain a vertex $x$ from one triangle, and a vertex $y$ from the other triangle. But then we automatically create a fourth triangle $\Delta = vxy$, in contradiction with the fact that there were supposed to be only $3$ triangles.\\
\indent Next we consider the case that we start with $2$ triangles having two vertices $u$ and $v$ in common. By (\ref{equation:prop2}) the third triangle cannot contain $u$ or $v$, hence must contain a vertex $x$ from one triangle and a vertex $y$ from the other triangle. But then $\{x,y,u,v\}$ induces $K_4$, contradicting (\ref{equation:prop1}). \\\\
\noindent \textit{Claim $5$: There are two triangles in $G$ that have exactly one vertex in common.}\\
Assume, to the contrary, that every pair of triangles in $G$ has two vertices in common (here we use (\ref{equation:prop2})). Let $\Delta_1 = abc$ and $\Delta_2 = abd$ be two triangles, with $a,b,c,d \in V$ and $c \neq d$, having the vertices $a$ and $b$ in common. By (\ref{equation:prop1}), $cd \notin E$. Then any triangle in $G$ contains at most one of $c,d$. Therefore, by our assumption that every pair of triangles has two vertices in common, any triangle must contain both $a$ and $b$. This contradicts (\ref{equation:prop2}). The claim follows.\\\\
\noindent By claims $3$ and $4$ there are exactly $4$ triangles in $G$. By claim $5$ there are two triangles that have exactly one vertex $v$ in common. A third triangle must share a vertex with both triangles, leading (up to relabeling the vertex $v$) to the following three possibilities:
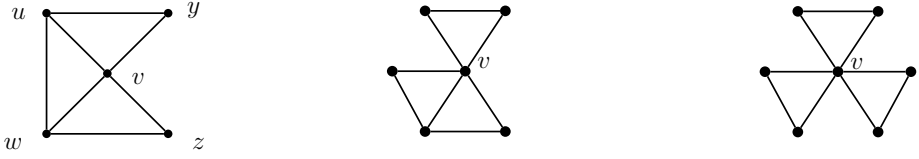
\begin{figure}[H]
\centering
\begin{subfigure}{0.3\linewidth}
  \scalebox{0.8}{
\begin{tikzpicture}
\SetGraphUnit{1}
\SetVertexNoLabel
\tikzset{VertexStyle/.style = {shape = circle,fill = black,minimum size = 4pt,inner sep=0pt}}
\Vertex{a}
\NOEA(a){b}
\NOWE(a){c}
\SOWE(a){d}
\SOEA(a){e}
\draw[style={-,thick,color=black}] (a)--(b)node[pos=0.8,anchor=south]{\hspace{13mm}{\large$y$}};
\draw[style={-,thick,color=black}] (a)--(c)node[pos=-0.4,anchor=south]{\hspace{4mm}{\large$v$}};
\draw[style={-,thick,color=black}] (b)--(c)node[pos=1.0,anchor=east]{\hspace{-10mm}{\large$u$}};
\draw[style={-,thick,color=black}] (a)--(d)node[pos=1.5,anchor=south]{\hspace{-3mm}{\large$w$}};
\draw[style={-,thick,color=black}] (a)--(e)node[pos=1.5,anchor=south]{\hspace{2mm}{\large$z$}};
\draw[style={-,thick,color=black}] (d)--(e)node[pos=0.5,anchor=north]{};
\draw[style={-,thick,color=black}] (c)--(d)node[pos=0.5,anchor=north]{};

\end{tikzpicture}}
\end{subfigure}
\begin{subfigure}{0.3\linewidth}
  \scalebox{0.8}{
\begin{tikzpicture}
\tikzset{VertexStyle/.style = {shape = circle,fill = black,minimum size = 5pt,inner sep=0pt}}
\SetVertexNoLabel
\Vertex{a}
\Vertex[x=-1.2,y=-0.0]{f}
\Vertex[x=-.66,y=-1]{g}
\Vertex[x=.66,y=1]{b}
\Vertex[x=-.66,y=1]{c}
\Vertex[x=.66,y=-1]{d}
\draw[style={-,thick,color=black}] (a)--(b)node[pos=-.2,anchor=south]{\hspace{7.5mm}{\large$v$}};
\draw[style={-,thick,color=black}] (a)--(c)node[pos=0.9,anchor=south]{};
\draw[style={-,thick,color=black}] (b)--(c)node[pos=0.5,anchor=north]{};
\draw[style={-,thick,color=black}] (a)--(d)node[pos=1.4,anchor=south]{};
\draw[style={-,thick,color=black}] (a)--(f)node[pos=0.5,anchor=north]{};
\draw[style={-,thick,color=black}] (a)--(g)node[pos=0.5,anchor=north]{};
\draw[style={-,thick,color=black}] (f)--(g)node[pos=0.5,anchor=north]{};
\draw[style={-,thick,color=black}] (d)--(g)node[pos=0.5,anchor=north]{};
\end{tikzpicture}}
\end{subfigure}
\begin{subfigure}{0.2\linewidth}
  \scalebox{0.8}{
\begin{tikzpicture}
\tikzset{VertexStyle/.style = {shape = circle,fill = black,minimum size = 5pt,inner sep=0pt}}
\SetVertexNoLabel
\Vertex{a}
\Vertex[x=-1.2,y=-0.0]{f}
\Vertex[x=-.66,y=-1]{g}
\Vertex[x=.66,y=1]{b}
\Vertex[x=-.66,y=1]{c}
\Vertex[x=.66,y=-1]{d}
\Vertex[x=1.2,y=-.0]{e}
\draw[style={-,thick,color=black}] (a)--(b)node[pos=-.2,anchor=south]{\hspace{7.5mm}{\large$v$}};
\draw[style={-,thick,color=black}] (a)--(c)node[pos=0.9,anchor=south]{};
\draw[style={-,thick,color=black}] (b)--(c)node[pos=0.5,anchor=north]{};
\draw[style={-,thick,color=black}] (a)--(d)node[pos=1.4,anchor=south]{};
\draw[style={-,thick,color=black}] (a)--(e)node[pos=0.5,anchor=south]{};
\draw[style={-,thick,color=black}] (d)--(e)node[pos=0.5,anchor=north]{};
\draw[style={-,thick,color=black}] (a)--(f)node[pos=0.5,anchor=north]{};
\draw[style={-,thick,color=black}] (a)--(g)node[pos=0.5,anchor=north]{};
\draw[style={-,thick,color=black}] (f)--(g)node[pos=0.5,anchor=north]{};
\end{tikzpicture}}
\end{subfigure}
\vspace{0mm}
\caption{From left to right: possibilities $\textrm{I}$, $\textrm{II}$ and $\textrm{III}$} 
\end{figure}
\noindent In possibilities $\textrm{II}$ and $\textrm{III}$ the fourth triangle must contain $v$, contradicting (\ref{equation:prop2}). In possibility $\textrm{I}$ the fourth triangle cannot contain the vertex $v$, and hence must contain the vertices $u$ and $w$. The fourth triangle does not contain $y$ or $z$, since that would create a copy of $K_4$. Hence the picture looks like:
\begin{figure}[H]
\centering
  \scalebox{0.6}{
\begin{tikzpicture}
\SetGraphUnit{1}
    \tikzset{VertexStyle/.style = {shape = circle,fill = black,minimum size = 6pt,inner sep=0pt}}
\SetVertexNoLabel
\Vertex{a}
\Vertex[x=1,y=1.5]{b}
\Vertex[x=2,y=3]{c}
\Vertex[x=3,y=1.5]{d}
\Vertex[x=2,y=0]{e}
\Vertex[x=4,y=0]{f}
\draw[style={-,thick,color=black}] (a)--(e)node[pos=0,anchor=east]{\hspace{-10mm}{\Large$x$}};
\draw[style={-,thick,color=black}] (a)--(b)node[pos=1.1,anchor=east]{\hspace{-8mm}{\Large$u$}};
\draw[style={-,thick,color=black}] (b)--(c)node[pos=1.25,anchor=east]{\hspace{-2mm}{\Large$y$}};
\draw[style={-,thick,color=black}] (b)--(d)node[pos=1.1,anchor=west]{\hspace{1mm}{\Large$v$}};
\draw[style={-,thick,color=black}] (b)--(e)node[pos=1.3,anchor=north]{\hspace{-3mm}{\Large$w$}};
\draw[style={-,thick,color=black}] (d)--(e)node[pos=0.5,anchor=north]{};
\draw[style={-,thick,color=black}] (e)--(f)node[pos=1,anchor=west]{\hspace{3mm}{\Large$z$}};
\draw[style={-,thick,color=black}] (c)--(d)node[pos=0.5,anchor=north]{};
\draw[style={-,thick,color=black}] (d)--(f)node[pos=0.5,anchor=north]{};
\end{tikzpicture}}
\vspace{0mm}
\caption{}\label{figure:4} 
\end{figure}
\noindent Let $a$ denote the seventh vertex of $G$. By connectedness and by (\ref{equation:prop0}), $a$ is adjacent to precisely one vertex. If $a$ is adjacent to one of $\{u,v,w\}$, then $\alpha(G) = 4$, contradicting (\ref{equation:prop1}). Hence $a$ is adjacent to one of $\{x,y,z\}$, yielding the graph $H_{26}$ from Figure \ref{allgraphs}.
\end{proof}

\subsection{Non-existence of forbidden graphs with at least $8$ vertices}\label{section:inclusion2}

 In this section we complete the proof of Theorem~\ref{27th}. In order to do so, we show that $\Forb_{n}$ is empty for $n \geq 8$. If $G = (V,E)$ is a graph, $U \subseteq V$ a subset of vertices and $x \in V$, then $N_U(x) := \{v \in U \mid xv \in E\}$. We write $\text{deg}_U(x) := |N_U(x)|$. 

\begin{lemma}\label{lemma:max2}
Let $G \in \mathrm{Forb}$. Then $\mathrm{max}\{\alpha(G),\omega(G)\} \leq 3$.
\end{lemma}
\begin{proof}
Let $G \in \Forb$. Then $\alpha(G)+ \omega(G) = |V(G)|-1$ by Corollary \ref{corollary:n-1}. We may assume, by going to the complement if necessary, that $\alpha(G) \geq \omega(G)$. We also may assume that $\alpha(G) \geq 4$, as there is nothing to prove if this were false. Then it follows that $G$ is $\FF$-free. Indeed, if $G$ contains a graph $H$ from $\FF$ as an induced subgraph, then by minimality $G = H$, contradicting the fact that $\alpha(H) \leq 3$ for all $H \in \FF$. We distinguish two cases: 
\begin{enumerate}
\item There is a maximum size clique $M$ and a maximum size stable set $S$ such that $M \cap S = \emptyset$.
\item For every maximum size clique $M$ and maximum size stable set $S$, $M \cap S \neq \emptyset$. 
\end{enumerate}
\noindent \textit{Case $1$}\\\\
\noindent We have $|M| + |S| = \omega(G) + \alpha(G) = |V(G)|-1$. Let $x$ be the vertex in $V(G) \setminus (M \cup S)$. If $x$ is not adjacent to any vertex in $S$, then $S \cup \{x\}$ is a stable set of size $\alpha(G) +1$. Hence, $x$ is adjacent to a vertex in $S$. We consider the cases: $|N_S(x)| = 1, |N_S(x)| = 2$ and $|N_S(x)| \geq 3$, beginning with the latter.\\

\noindent \textit{$1.1$ $|N_S(x)| \geq 3$}\\\\
\noindent Let $s \in S$. By the assumption on the cardinality of $N_S(x)$, a maximum size stable set in $G - s$ does not contain the vertex $x$. Hence it must contain a vertex $m_s \in M$. This implies that $m_s$ is not adjacent to any vertex in $S\setminus\{s\}$, but is adjacent to $s$ (otherwise $S\cup\{m_s\}$ is a stable set of size $\alpha(G)+1$). This in turn implies that $m_s \neq m_{s'}$, if $s$ and $s'$ are distinct vertices in $S$. Hence, $\omega(G) \geq \alpha(G)$ and therefore $\alpha(G) = \omega(G) \geq 4$.\\
\indent Let $m$ be a non-neighbor of $x$ in $M$. It exists, as otherwise $M \cup \{x\}$ induces a clique of size $\omega(G)+1$. Let $m' \in M\setminus \{m\}$. As $\text{deg}_M(s) = 1$ for all $s \in S$, a maximum size clique in $G-m'$ cannot contain a vertex from $S$ (here we use that $\omega(G) \geq 4$). As the vertex $x$ is not adjacent to $m$, the vertices $(M\setminus m')\cup \{x\}$ do not form a clique. Then $\omega(G-m') < \omega(G)$, contradicting Lemma \ref{lemma:removal}. Hence, this case cannot occur.\\

\noindent \textit{$1.2$ $|N_S(x)| = 2$}\\\\
\noindent Let $N_S(x) = \{a,b\}$ and let $s \in S\setminus \{a,b\}$. If a stable set in $G-s$ contains $x$, then it cannot contain $a$ and $b$ and hence it is not of maximum size. Therefore, a maximum size stable set in $G-s$ does not contain $x$, but contains a vertex $m_s \in M$. As in the previous subcase, $m_s$ is non-adjacent to every vertex in $S\setminus \{s\}$, but is adjacent to $s$. Furthermore, $m_s \neq m_{s'}$, if $s$ and $s'$ are distinct vertices in $S$, showing that $\omega(G) \geq \alpha(G)-2$. \\
\indent As $b$ and $x$ are adjacent, a maximum size stable set in $G-a$ necessarily contains a vertex $m_a \in M$. The vertex $m_a$ is non-adjacent to every vertex in $S\setminus \{a,b\}$. Hence $m_a \neq m_s$, for each $s \in S\setminus \{a,b\}$. This gives $\omega(G) \geq \alpha(G)-1$ and also that 
\begin{equation}\label{equation:degreebound}
\text{deg}_M(s) \leq 2,
\end{equation}
for $s \in S\setminus\{a,b\}$. The vertex $m_a$ is non-adjacent to at least one of $b,x$, as one of $b,x$ is contained in a maximum size stable set in $G-a$ that contains $m_a$. This also implies that $m_a$ is adjacent to a vertex in $\{a,b\}$ (otherwise $S\cup\{m_a\}$ is a stable set of size $\alpha(G)+1$). Without loss of generality, assume that $m_a$ is adjacent to $a$. We make a case distinction.\\

\noindent \textit{$1.2.1$ $m_a$ is not adjacent to $b$, but it is adjacent to $x$}\\\\
\noindent Let $S'$ be a maximum size stable set in $G-b$. Suppose $S'$ contains $x$. Then it contains neither $a$, nor $m_a$ (as it is assumed to be adjacent to $x$) nor any vertex in $M$ of the form $m_s$, with $s \in S\setminus \{a,b\}$. Hence there must be an additional vertex $m_b \in M$ that is in $S'$. Then $S' = (S\setminus \{a,b\}) \cup \{x,m_b\}$. Now consider the graph $G-m_a$. As $m_s$ is non-adjacent to every vertex in $S\setminus\{s\}$, for $s \in S\setminus\{b\}$, the vertices $a$ and $b$ have no neighbors in $M\setminus \{m_a,m_b\}$. Furthermore, $x$ is not adjacent to $m_b$. Observe that there is no clique of size $\omega(G)$ in $G-m_a$, contradicting Lemma \ref{lemma:removal}.\\
\indent Hence $x \notin S'$. Then $S' = (S \setminus \{b\}) \cup \{m_b\}$, for some $m_b \in M$ that is neither $m_a$ (as $m_a$ is adjacent to $a$), nor $m_s$, for some $s \in S\setminus\{a,b\}$ (as $m_s$ is adjacent to $s$). Hence $\omega(G) = \alpha(G) \geq 4$. It also follows that $N_M(v) = \{m_v\}$, for $v \in \{a,b\}$. Together with (\ref{equation:degreebound}) this yields $\text{deg}_M(s) = 1$, for all $s \in S$. This in turn implies that no vertex in $S$ can be in a maximum size clique in $G-m_a$. A maximum size clique in $G-m_a$ also cannot contain $x$, as $x$ has a non-neighbor in $M$ that is not $m_a$, by assumption. Therefore $\omega(G-m_a) < \omega(G)$, contradicting Lemma \ref{lemma:removal}.\\

\noindent \textit{$1.2.2$ $m_a$ is not adjacent to $b$, and not adjacent to $x$}\\\\
\noindent Take a vertex of the form $m_s \in M$, for some $s \in S\setminus\{a,b\}$. Either $m_sx \in E$, in which case $\{a,b,x,s,m_s,m_a\}$ induces $H_9$ from Figure \ref{allgraphs}, contradicting the fact that $G$ was supposed to be $\FF$-free. Or $m_sx \notin E$, in which case $\{a,b,x,s,m_s,m_a\}$ induces $P_6$ (which is $H_6$ in Figure \ref{allgraphs}).\\

\noindent \textit{$1.2.3$ $m_a$ is adjacent to $b$, but not adjacent to $x$}\\\\
\noindent Take a vertex of the form $m_s \in M$, for some $s \in S\setminus\{a,b\}$. Either $m_sx \in E$, in which case $\{a,b,x,s,m_s,m_a\}$ induces $H_{11}$. Or $m_sx \notin E$, in which case $\{a,b,x,s,m_s,m_a\}$ induces $H_7$.\\

\noindent \textit{$1.3$ $|N_S(x)| = 1$}\\\\
\noindent Let $N_S(x) = \{a\}$ and let $s \in S\setminus \{a\}$. A maximum size stable set in $G-s$ either contains $a$ or $x$, but not both. Hence, it contains exactly one vertex $m_s$ from $M$. Then $m_s$ is non-adjacent to every vertex in $S\setminus\{a,s\}$. We argue that $m_s$ is adjacent to $s$. Assume that $m_ss \notin E$. Since $S \cup \{m_s\}$ is not a stable set, we must have that $m_sa \in E$. Also, $m_sx\in E$ (otherwise $(S\setminus\{a\}) \cup \{m_s,x\}$ is a stable set of size $\alpha(G)+1$). Observe that there is no stable set of size $\alpha(G)$ in $G-s$, contradicting Lemma \ref{lemma:removal}. Hence $m_s$ is adjacent to $s$. The argument also shows that $m_s \neq m_{s'}$, if $s$ and $s'$ are distinct vertices in $S\setminus\{a\}$, implying that $\omega(G) \geq \alpha(G)-1 \geq 3$. We consider the following possibilities.\\

\noindent \textit{$1.3.1$ $N_M(a) = \emptyset$ and $N_M(x) = \emptyset$}\\\\
\noindent Let $s_1,s_2$ be two distinct vertices in $S \setminus \{a\}$ (they exist, as $\alpha(G) \geq 4$). Then $\{s_1,s_2,a,x,m_{s_1},m_{s_2}\}$ induces $H_3 (= P_4 + K_2)$. \\

\noindent \textit{$1.3.2$ $N_M(a) = \emptyset$ and $N_M(x) \neq \emptyset$}\\\\
\noindent Let $m_x$ be a neighbor of $x$ in $M$. All vertices in $M$ have a neighbor in $S$, otherwise there is a stable set of size $\alpha(G)+1$. As $a$ is non-adjacent to every vertex in $M$, we know that $m_x$ has a neighbor $s_1 \in S\setminus\{a\}$. Let $m_x'$ be a non-neighbor of $x$ in $M$. Assume that the neighbor $s_2$ of $m_x'$ in $S\setminus\{a\}$ is not $s_1$. Then $\{s_1,s_2,a,x,m_x,m_x'\}$ induces $H_4$.\\
\indent Suppose now that $s_2 = s_1$. A maximum size clique in $G-m_x$ must contain $x$. Indeed, if it does not, then it must be of the form $(M\setminus \{m_x\})\cup\{s\}$, for some $s \in S\setminus\{a\}$. But $\omega(G) \geq 3$ and $\text{deg}_{M\setminus\{m_x\}}(s) \leq 1$, for all $s \in S\setminus\{a\}$. So a maximum size clique in $G-m_x$ contains $x$, and then it necessarily does not contain $m_x'$. Then it must contain a vertex from $S$, namely $a$ (as $x$ is adjacent only to $a$ in $S$). But $a$ has no neighbors in $M$ by assumption, a contradiction.\\

\noindent \textit{$1.3.3$ $N_M(a) \neq \emptyset$ and $N_M(x) = \emptyset$}\\\\
\noindent This case is reduced to case $1.3.2$ by interchanging the role of $a$ and $x$.\\

\noindent \textit{$1.3.4a$ $N_M(a) \cap N_M(x) \neq \emptyset$}\\\\
\noindent Let $m$ be a common neighbor of $a$ and $x$ in $M$. Observe that $m \neq m_s$, for $s \in S\setminus\{a\}$ (otherwise $(S\setminus\{s\})\cup\{m_s\}$ is not a stable set). Hence $\omega(G) = \alpha(G) \geq 4$. Then a maximum size clique in $G-m$ uses no vertex of $S\setminus \{a\}$ (as $\text{deg}_M(s) \leq 2$, for all $s \in S\setminus\{a\}$). As both $a$ and $x$ have a non-neighbor in $M$, a maximum size clique in $G-m$ is of the form $(M\setminus \{p,m\}) \cup \{a,x\}$. Here, $p$ is a common non-neighbor of $a$ and $x$ in $M$ and it is the only non-neighbor of $a$, and of $x$. But then $(M\setminus \{p\}) \cup \{a,x\}$ is a clique of size $\omega(G)+1$ in $G$, a contradiction.\\

\noindent \textit{$1.3.4b$ $N_M(a) \neq \emptyset$ and $N_M(x) \neq \emptyset$ but $N_M(a) \cap N_M(x) = \emptyset$}\\\\
\noindent Let $m_a$ be a neighbor of $a$ in $M$ and let $m_x$ be a neighbor of $x$ in $M$. Assume that $\omega(G) = 3$. Then there are $s,s' \in S\setminus\{a\}$ such that $m_s = m_a$ and $m_{s'} = m_x$. By the assumption that $a$ and $x$ have no common neighbor, $\{s,s',a,x,m_a,m_x\}$ induces $H_9$. \\
\indent Hence $\omega(G) \geq 4$. Suppose a maximum size clique in $G-m_a$ contains a vertex from $S\setminus \{a\}$. Since $\text{deg}_M(s) \leq 2$, for all $s \in S\setminus\{a\}$, it must also contain $a$ or $x$ (here we use that $\omega(G) \geq 4$). But both $a$ and $x$ have no other neighbors in $S\setminus\{a\}$. A maximum size clique in $G-m_a$ also cannot contain both $a$ and $x$, as they do not have common neighbors in $M$. The set $(M\setminus \{m_a\}) \cup \{a\}$ is no clique as $m_xa \notin E$. Hence a maximum size clique in $M-m_a$ is of the form $(M\setminus \{m_a\}) \cup \{x\}$, implying that $m_a$ is the only non-neighbor of $x$ in $M$ and also that $N_M(a) = \{m_a\}$. The same arguments as before show that a maximum size clique in $G-m_x$ neither contains both $a$ and $x$, nor a vertex from $S\setminus \{a\}$. The set $(M\setminus \{m_x\})\cup \{a\}$ is not a clique, as $N_M(a) = \{m_a\}$. But also $(M\setminus m_x)\cup \{x\}$ is not a clique, as $m_ax \notin E$. Hence $\omega(G-m_x) < \omega(G)$, a contradiction to Lemma \ref{lemma:removal}.\\

\noindent \textit{Case $2$}\\

\noindent Let $S$ be a maximum size stable set and $M$ be a maximum size clique. By assumption $S \cap M \neq \emptyset$, hence $|S \cap M| = 1$. Let $x \in S \cap M$. Since $|V(G)| -1 = \alpha(G) + \omega(G)$ there are exactly two vertices in $V(G)\setminus (S \cup M)$, say $y$, $z$. As we are in case $2$, $(S\setminus\{x\})\cup\{y,z\}$ cannot contain a stable set of size $\alpha(G)$. Hence each of $y,z$ has a neighbor in $S\setminus\{x\}$. We consider two cases.  \\

\noindent \textit{$2.1$ Both $y,z$ have exactly one neighbor in $S\setminus\{x\}$, which is a common neighbor.}\\\\
\noindent Let $a \in S\setminus\{x\}$ be the common neighbor of $x$ and $y$. Note that in this case $y$ is adjacent to $z$, for otherwise $(S\setminus\{a,x\})\cup\{y,z\}$ is a stable set of size $\alpha(G)$ disjoint from $M$, which is impossible. Since $|S| \geq 4$, there is a vertex $s \in S\setminus\{x\}$ that is distinct from $a$. Now a stable set in $G - s$ can contain at most one vertex of $M$, and one of $\{a,y,z\}$. Thus $\alpha(G-s) \leq  \alpha(G)-3 + 1 + 1 = \alpha(G) - 1$, a contradiction to Lemma \ref{lemma:removal}. \\

\noindent \textit{$2.2$ There exist distinct vertices $a,b \in S\setminus\{x\}$ such that $a$ is adjacent to $y$ and $b$ is adjacent to $z$.}\\\\
\noindent Since $|S| \geq 4$, there is a vertex $s \in S\setminus\{x\}$ that is distinct from $a,b$. Now a stable set in $G - s$ can contain at most one vertex of $M$, and at most one of $a,y$ and at most one of $b,z$. Thus $\alpha(G-s) \leq  \alpha(G)-4 + 1 + 1 + 1 = \alpha(G) - 1$, contradicting Lemma \ref{lemma:removal}. 
\end{proof}

\begin{lemma}\label{lemma:bound}
If $G \in \mathrm{Forb}$, then $|V(G)| \leq 7$.
\end{lemma}
\begin{proof}
By Corollary \ref{corollary:n-1} and Lemma \ref{lemma:max2} we have $|V(G)| = \alpha(G) + \omega(G) + 1 \leq 3 + 3 + 1 = 7$.
\end{proof}

\begin{proof}[Proof of Theorem \ref{27th}]
We have to prove that $\Forb = \FF$, or equivalently, that $\Forb_n = \FF_n$, for all $n \geq 1$. For $n \leq 7$, this is the content of Lemmas \ref{lemma:6} and \ref{lemma:7}. By Lemma \ref{lemma:bound}, $\Forb_n$ is empty for $n \geq 8$. Hence, we are done. 
\end{proof}

\section{Future directions}\label{section:add}

All but three graphs~$H \in \mathcal{F}$ have the property that~$H$ or~$\overline{H}$ is bipartite, has six vertices and has a perfect matching.  It seems that the three graphs in~$\FF$ that do not satisfy this property (i.e., the 5-cycle~$H_1=C_5$ and the complementary seven-vertex graphs~$H_{26}$ and~$H_{27}$) are not the most important excluded induced subgraphs for obtaining a high lower bound on~$\alpha+\omega$. Write~$\mathcal{B}:=  \mathcal{F} \setminus \{H_1, H_{26},H_{27}\} =\{H_2,\ldots H_{25}\}$. We believe that the following, which was verified by computer to be true for all graphs with at most~$10$ vertices, holds:
\begin{conjecture}
Every~$\mathcal{B}$-free graph~$G$ satisfies $\alpha(G) + \omega(G) \geq |V(G)| - 1$. 
\end{conjecture}
\noindent Another problem that naturally arises from the main theorem is the following. For a positive integer $c$, let $\mathcal{H}_c$ denote the class of graphs $G$ such that every induced subgraph $H$ of $G$ satisfies $\alpha(H) + \omega(H) \geq |V(H)|-c$. Is the list of forbidden induced subgraphs for $\mathcal{H}_c$ finite?  (Note that $\mathcal{H}_1$ is the class of sum-perfect graphs.) For $c=1$, the list of forbidden induced subgraphs with at most~$8$ vertices already contains~$>1000$ members, as was found by computer.

From the definition of sum-perfect graphs and  Lov\'{a}sz' characterization of perfect graphs, it is easy to see that sum-perfect graphs are perfect. A graph $G$ is \emph{strongly perfect} if every induced subgraph $H$ of $G$ contains a stable set that intersects all the maximal (with respect to inclusion) cliques of $H$ (see \cite{BD}). Are sum-perfect graphs strongly perfect?

The class of sum-perfect graphs gives rise to interesting algorithmic questions. The problems STABLE SET, MAXCLIQUE, COLORING, CLIQUE COVER are all polynomial for sum-perfect graphs because sum-perfect graphs are weakly chordal, and fast algorithms are known for the latter (see \cite{HHM}, \cite{HSS}). Are there faster algorithms by exploiting the special structure of sum-perfect graphs? 

A last algorithmic question is the problem of recognizing sum-perfect graphs. The main theorem of this paper implies that there is a $O(n^7)$ algorithm to recognize sum-perfect graphs with~$n$ vertices: just test whether any of the~$27$ graphs from~$\FF$ appears as an induced subgraph. Since the largest of these graphs has $7$ vertices, we get an algorithm whose running time is $O(n^7)$. Is there a faster algorithm to recognize sum-perfect graphs? These questions are material for further research.\\

\noindent \textit{Acknowledgements.} The authors would like to thank Bart Sevenster for helpful discussions.

\selectlanguage{english}

\end{document}